\documentclass[journal,twoside,web]{ieeecolor}
\usepackage{generic}
\usepackage{cite}
\usepackage{amsmath,amssymb,amsfonts}
\usepackage{hyperref}
\usepackage[ruled,vlined,linesnumbered]{algorithm2e}
\usepackage{algpseudocode}
\usepackage{graphicx}
\usepackage{textcomp}
\def\BibTeX{{\rm B\kern-.05em{\sc i\kern-.025em b}\kern-.08em
    T\kern-.1667em\lower.7ex\hbox{E}\kern-.125emX}}

\newtheorem{rmk}{Remark}[section]
\newtheorem{thm}{Theorem}[section]
\newtheorem{lmm}{Lemma}[section]

\usepackage{array}
\usepackage{tabularx}
\newcolumntype{C}{ >{\centering} m{15cm} }
\newcolumntype{G}{>{\centering\arraybackslash}p{4mm}}
\newcolumntype{S}{>{\centering\arraybackslash\scriptsize}p{4mm}}

\newcommand\Tstrut{\rule{0pt}{2.6ex}}         
\newcommand\Bstrut{\rule[-0.9ex]{0pt}{0pt}}   
\usepackage{multirow}

\newcommand{\Unif}{\textrm{Unif}}

\let\labelindent\relax
\usepackage{enumitem}
\newenvironment{Itemize}{\begin{itemize}[left=3ex]}{\end{itemize}}

\definecolor{lightblue}{rgb}{0.36, 0.54, 0.7}

\newcommand{\revThird}[1]{\textcolor{black}{#1}}
\newcommand{\revFourth}[1]{\textcolor{black}{#1}}

\begin{document}
\title{Analysis of Theoretical and Numerical Properties of Sequential Convex Programming for Continuous-Time Optimal Control}
\author{Riccardo Bonalli, Thomas Lew, and Marco Pavone
%
\thanks{This work was supported in part by NASA under the Early Stage Innovations program, by the National Science Foundation under the CPS program (grant \#1931815), and by the King Abdulaziz City for Science and Technology (KACST).}
\thanks{R. Bonalli is with the Laboratory of Signals and Systems (L2S), Universit\'e Paris-Saclay, Centre National de la Recherche Scientifique (CNRS), CentraleSup\'elec, 91190 Gif-sur-Yvette, France (emails: riccardo.bonalli@centralesupelec.fr and rbonalli@stanford.edu).}
\thanks{T. Lew and M. Pavone are with the Department of Aeronautics \& Astronautics, Stanford University, Stanford, CA 94305-4035 USA (emails: thomas.lew@stanford.edu and pavone@stanford.edu).} 
}

\maketitle

\begin{abstract}
Sequential Convex Programming (SCP) has recently gained significant popularity as an effective method for solving optimal control problems and has been successfully applied in several different domains. However, the theoretical analysis of SCP has received comparatively limited attention, and it is often restricted to discrete-time formulations. In this paper, we present a unifying theoretical analysis of a fairly general class of SCP procedures for continuous-time optimal control problems. In addition to the derivation of convergence guarantees in a continuous-time setting, our analysis reveals two new numerical and practical insights. 
First, we show how one can more easily account for manifold-type constraints, which are a defining feature of optimal control of mechanical systems. Second, we show how our theoretical analysis can be leveraged to accelerate SCP-based optimal control methods by infusing techniques from indirect optimal control. 

\end{abstract}



\section{Introduction}
\label{sec:introduction}
\IEEEPARstart{S}{ince} its first appearance more than five decades ago, Sequential Convex Programming (SCP) \cite{falk1969,mccormick1976} has proven to be a powerful and reliable algorithmic framework for non-convex optimization, and it has recently gained new popularity in aerospace \cite{LiuLu2014,morgan2014,MaoSzmukEtAl2016,Virgili-llopZagarisEtAl2017} and robotics \cite{AugugliaroSchoelligEtAl2012,SchulmanDuanEtAl2014,BonalliCauligiEtAl2019,BonalliBylardEtAl2019}. In its most general form, SCP entails finding a locally-optimal solution to a non-convex optimization problem as the limit point of a sequence of solutions to convex subproblems formed by successive approximations. The main advantage offered by this approach is the ability to leverage a wide spectrum of numerical techniques to efficiently solve each convex subproblem \cite{BoydVandenberghe2004,azhmyakov2008,VerscheureDemeulenaereEtAl2009,chamanbaz2015}, leading to near-real-time numerical schemes. For example, among the most mature SCP paradigms we find the well-known Sequential Quadratic Programming (SQP) method \cite{betts1993,sideris2005,NocedalWright2006}.

Through the years, SCP's sound performance has pushed the community towards deep investigations of the theoretical nature of this method. The most 
informative result states that when convergence is achieved, SCP finds a candidate local optimum for the original non-convex problem, i.e., a solution that satisfies necessary conditions for local optimality \cite{zillober2009,DinhDiehl2010,diehl2019} (convergence rate results have also been derived, see, e.g., \cite{messerer2020}). When used in the context of non-convex optimal control, the SCP convexification scheme is usually applied to the non-convex program that stems from a discretization of the original continuous-time problem, providing only partial insights with respect to the original continuous-time formulation. For instance, are those guarantees only applicable to specific discretization schemes? Can insights from continuous-time analysis be leveraged to improve SCP-based optimal control methods? To the best of our knowledge, the only continuous-time analysis of SCP-based optimal control is provided in \cite{MaoSzmukEtAl2016}, though the optimal control context considered by the authors is very specific and the conditions for optimality used are weaker than those in the state-of-the-art for continuous-time optimal control (see our discussion in Section \ref{sec_guarantees}).


\textit{Statement of contributions:} 
In this paper we contribute to filling the existing gap in the theoretical analysis of SCP-based optimal control methods by providing a unifying analysis of a wide class of SCP procedures for continuous-time (non-convex) optimal control. Our main result consists of proving that, under \revThird{mild assumptions}, any accumulation point for the sequence of solutions returned by SCP satisfies the \textit{Pontryagin Maximum Principle} (PMP) \cite{Pontryagin1987,AgrachevSachkov2004} associated with the original formulation. The PMP represents a set of necessary conditions for optimality in continuous-time optimal control that is stronger than the traditional Lagrange multiplier rules (the latter were investigated in \cite{MaoSzmukEtAl2016}), and it often represents the best result one might hope for in nonlinear optimal control. Our convergence result stems from an analysis on the continuity with respect to convexification of the Pontryagin cones of variations, tools originally introduced by Pontryagin and his group to prove the PMP. In addition, we relax some technical assumptions that are often difficult to verify in practice and that have been considered in \cite{MaoSzmukEtAl2016} (e.g., strong compactness of the set of admissible controls is replaced by weak compactness), thus enlarging the class of problems that can be solved by SCP with guarantees. 

Our continuous-time analysis provides a generalization of several existing discrete-time results and reveals new insights into the nature of SCP applied to optimal control, ultimately offering three key advantages. First, we can transfer theoretical guarantees to \textit{any} discrete-time implementation of the continuous-time SCP-based optimal control formulation, regardless of the time-discretization scheme adopted. Second, we can directly and effectively extend these guarantees to the setting with \textit{manifold-type constraints}, i.e., nonlinear state equality constraints often found when dealing with mechanical systems. Third, we can provide a powerful connection to indirect methods for optimal control such as \textit{(indirect) shooting methods} \cite{trelat2012}, enabling the design of numerical schemes that accelerate the convergence of SCP.

Specifically, our contributions are as follows: (1) We derive theoretical guarantees for continuous-time SCP-based optimal control methods, whose related sequence of convex subproblems stems from the successive linearization of all nonlinear terms in the dynamics and all non-convex functions in the cost. In particular, we apply this analysis to finite-horizon, finite-dimensional, non-convex optimal control problems with control-affine dynamics. (2) Through a study of the continuity of the Pontryagin cones of variations with respect to linearization, we prove that whenever the sequence of SCP iterates converges (under specific topologies), we find a solution satisfying the PMP associated with the original formulation. In addition, we prove that up to some subsequence, the aforementioned sequence always has an accumulation point, which provides a weak guarantee of success for SCP (``weak'' in the sense that only a \textit{subsequence} of the sequence of SCP iterates can be proved to converge). (3) We leverage the continuous-time analysis to design a novel and efficient approach to account for manifold-type constraints. Specifically, we show that, under mild assumptions, one can solve the original formulation (i.e., with manifold-type constraints) with convergence guarantees by applying SCP to a new optimal control problem where those constraints are simply ignored, thereby simplifying numerical implementation. (4) As a byproduct, our analysis shows that the sequence of multipliers associated with the sequence of convex subproblems converges to a multiplier for the original formulation. We show via numerical experiments how this property can be used to considerably accelerate convergence rates by infusing techniques from indirect control.  

Previous versions of this work have appeared in \cite{BonalliCauligiEtAl2019,BonalliBylardEtAl2019}. In this paper, we provide as additional contributions (i) a new formulation with more general cost functionals, (ii) convergence proofs under weaker assumptions, (iii) detailed explanations on ``transferring" theoretical guarantees under time discretizations, and (iv) extensive numerical simulations for the acceleration procedure based on indirect methods.

We do highlight three main limitations of our work. First, being SCP a \textit{local optimization algorithm}, our theoretical guarantees are necessarily local (this is arguably unavoidable given the local nature of SCP). Second, the assumption of control-affine dynamics plays a crucial (though technical) role in our convergence analysis. The extension of our results to the more general setting represents an open research question. 

\textit{Organization}: The paper is organized as follows. Section \ref{sec_probForm} introduces notation and the continuous-time non-convex optimal control problem we wish to study. Our convergence analysis of SCP-based optimal control methods is split in two sections: In Section \ref{sec_noManif}, convergence is analyzed in the absence of manifold-type constraints, and in Section \ref{sec_manif} we account for manifold-type constraints. We show in Section \ref{sec_acceleratingProc} how our theoretical analysis can be used to design convergence acceleration procedures through numerical experiments in Section \ref{sec_numerics}. Finally, Section \ref{sec_conclusion} provides final remarks and directions for future research.

\section{Problem Formulation} \label{sec_probForm}

Our objective consists of providing locally-optimal solutions to Optimal Control Problems (OCP) of the form:
\begin{equation*}
    \begin{cases}
        \displaystyle \underset{\revThird{u \in \mathcal{U}}}{\min} \ \int^{t_f}_0 f^0(s,x(s),u(s)) \; \mathrm{d}s \triangleq \int^{t_f}_0 \bigg( G(s,u(s)) \\
        \displaystyle \quad + H(s,x(s)) + L^0(s,x(s)) + \sum^m_{i=1} u^i(s) L^i(s,x(s)) \bigg) \; \mathrm{d}s \\
        \dot{x}(s) = f(s,x(s),u(s)) \\
        \displaystyle \qquad \triangleq f_0(s,x(s)) + \sum^m_{i=1} u^i(s) f_i(s,x(s)) , \ \textnormal{a.e.} \ \revThird{s \in [0,t_f]} \\
        x(0) = x^0 , \quad g(x(t_f)) = 0 \\
        x(s) \in M \subseteq \mathbb{R}^n , \quad s \in [0,t_f]
    \end{cases}
\end{equation*}
where the variable $x$ denotes state variables, and \revThird{we optimize over controls $u \in \mathcal{U} \triangleq L^2([0,t_f];U)$, where $t_f > 0$ is some fixed final time and $L^2([0,t_f];U)$ is the space of square integrable controls defined in $[0,t_f]$ and with image in $U$, $U \subseteq \mathbb{R}^m$ being a convex compact subset}. The set $\mathcal{U}$ contains all the \textit{admissible controls}. The mappings $L^i : \mathbb{R}^{n+1} \rightarrow \mathbb{R}$, $f_i : \mathbb{R}^{n+1} \rightarrow \mathbb{R}^n$, for $i=0,\dots,m$, and $g : \mathbb{R}^n \rightarrow \mathbb{R}^{\ell_g}$ are assumed to be smooth (i.e., at least twice continuously differentiable), whereas we consider smooth mappings $G : \mathbb{R}^{m+1} \rightarrow \mathbb{R}$, $H : \mathbb{R}^{n+1} \rightarrow \mathbb{R}$ that are convex with respect to the variables $u$ and $x$, respectively. We require that the vector fields $f_i$, $i=0,\dots,m$ have compact supports \revThird{(or alternatively that $f_i$, $i=0,\dots,m$ and their first and second derivatives with respect to $x$ are bounded)}, and that 
$0$ is a regular value for $g$, so that $g^{-1}(0)$ is a submanifold of $\mathbb{R}^n$, and that $g(x_0) \neq 0$, so that no trivial solutions exists. In addition, we may require optimal trajectories to satisfy \textit{manifold-type constraints} of the form $x(s) \in M$, $s \in [0,t_f]$, where $M \subseteq \mathbb{R}^n$ is a smooth $d$-dimensional submanifold of $\mathbb{R}^n$. In this case, the initial condition $x^0 \in \mathbb{R}^n$ lies within $M$. In OCP, the mappings $f$ and $f^0$ model \textit{control-affine} nonlinear dynamics which are satisfied almost everywhere (a.e) in \revThird{$[0,t_f]$}, and non-convex-in-state cost, respectively. In particular, we leverage the fact that controls $u$ appear in the cost $f^0$ through either convex or linear terms only to establish convergence guarantees. 
Any (locally-optimal) solution to OCP is denoted as \revThird{$(x^*,u^*)$, where the control $u^* : [0,t_f] \rightarrow U$ is in $L^2([0,t_f];U)$ and $x^* : [0,t_f] \rightarrow \mathbb{R}^n$ is an absolutely-continuous trajectory}.

\begin{rmk}
    The requirement that the vector fields $f_i$, $i=0,\dots,m$ have compact supports is not restrictive in practice, for we may multiply the $f_i$ by some smooth cut-off function whose support is in some arbitrarily large compact set that contains states $x \in \mathbb{R}^n$ which are relevant to the given application domain. 
    Importantly, as a standard result, this property implies that the trajectory solutions to the dynamics of OCP (and to the dynamics of every other problem that will be defined later) are defined and uniformly bounded for all times \revThird{$s \in [0,t_f]$}, see Lemma \ref{lemma_boundness} for a more precise statement. From this last observation and Filippov's theorem, we infer the existence of (at least locally) optimal solutions to OCP as long as OCP is feasible (see, e.g., \cite{LeeMarkus1967}). Sufficient conditions for the feasibility of OCP exist and are related to the Lie algebra generated by $f_1,\dots,f_m$. In particular, these conditions are \textit{generic} (more details may be found in \cite{lobry1972}; see also Section \ref{sec_convexProb}). 
\end{rmk}

\begin{rmk} \label{remark_penalization}
Many applications of interest often involve state constraints $c(s,x(s)) \le 0$, $s \in [0,t_f]$, where the mapping $c : \mathbb{R}^{n+1} \rightarrow \mathbb{R}^{\ell_c}$ is smooth and non-convex. One common way of solving such constrained problems hinges on the \textit{penalization} of state constraints within the cost, thus reducing the original problem to OCP. Specifically, given a penalization weight $\omega = (\omega_1,\dots,\omega_{\ell_c}) \in [0,\omega_{\max}]^{{\ell_c}}$, one may introduce the mapping $L^0_{\omega}(s,x) \triangleq L^0(s,x) + \sum^{\ell_c}_{i=1} \omega_i h(c_i(s,x))$, where $h : \mathbb{R} \rightarrow \mathbb{R}_+$ is any continuously differentiable penalization function such that $h(z) = 0$ for $z \le 0$ (e.g., $h(z) = 0$ for $z \le 0$ and $h(z) = z^2$ for $z > 0$). The constrained problem is reduced to OCP by dropping state constraints and replacing the running cost function $L_0$ with $L^0_{\omega}$ (note that $L^0_{\omega}$ is smooth but not necessarily convex). The parameter $\omega$ is selected by the user and weighs the presence of state constraints; the higher the value, the larger the penalization for the violation of state constraints. We will use this remark for numerical experiments in Section \ref{sec_numerics}. We refer to \cite{NocedalWright2006} for the analysis of the convergence for $\omega \rightarrow \infty$ of penalty methods toward solutions of constrained optimization problems, which lies outside the scope of this work.
\end{rmk}

\revThird{OCP} is in general difficult to solve because of the presence of nonlinear dynamics and non-convex cost. The solution strategy proposed in this work is based on SCP.

\section{Sequential Convex Programming without\\ Manifold-Type Constraints} \label{sec_noManif}

As a first step, we develop our SCP framework without considering manifold-type constraints, showing later how the whole formalism can be adapted to the presence of those constraints. Dropping the manifold-type constraints, OCP takes the simpler form:
\begin{equation*}
    \begin{cases}
        \displaystyle \underset{\revThird{u \in \mathcal{U}}}{\min} \ \int^{t_f}_0 f^0(s,x(s),u(s)) \; \mathrm{d}s \\
        \dot{x}(s) = f(s,x(s),u(s)) , \ \textnormal{a.e.} \ \revThird{s \in [0,t_f]} \\
        x(0) = x^0 , \quad g(x(t_f)) = 0 .
    \end{cases}
\end{equation*}
SCP entails finding a locally-optimal solution to OCP as a limit point of a sequence of solutions to convex subproblems coming from successive approximations to OCP. Although several different approximation schemes have been introduced in the literature, in this work we focus on arguably the simplest one, which is to linearize any nonlinear term in the dynamics and any non-convex function in the cost. The two main advantages of this approach are ease of computing linearizations and the absence of high-order singular Jacobians, which can cause the SCP problem to be ill-posed (e.g., SQP requires additional procedures to ensure positive definiteness of Hessians \cite{NocedalWright2006}). 

\subsection{Design of Convex Subproblems} \label{sec_convexProb}

Assume we are given \revThird{$(x_0,u_0)$, where $u_0 : [0,t_f] \rightarrow \mathbb{R}^m$ is piecewise continuous and $x_0 : [0,t_f] \rightarrow \mathbb{R}^n$ is absolutely continuous}. This tuple represents the initializing guess for the SCP procedure. Importantly, we do not require $(x_0,u_0)$ to be feasible for OCP, though feasibility of $(x_0,u_0)$ and closeness to a satisfactory trajectory increases the chances of rapid convergence \revThird{(as it was empirically observed, e.g., in \cite{BonalliCauligiEtAl2019})}. We will address this point further in the numerical experiment section. A sequence of \textit{convex optimal control problems} is defined by induction as follows: Given a sequence $(\Delta_k)_{k \in \mathbb{N}} \subseteq \mathbb{R}_+$, the Linearized Optimal Control subProblem (LOCP$^{\Delta}_{k+1}$) at iteration $k+1$ subject to trust-region radius $\Delta_{k+1} > 0$ is

{\small
\begin{equation*}
    \begin{cases}
        \displaystyle \underset{\revThird{u \in \mathcal{U}}}{\min} \ \int^{t_f}_0 f^0_{k+1}(s,x(s),u(s)) \; \mathrm{d}s \\
        \displaystyle \qquad \triangleq \int^{t_f}_0 \bigg( G(s,u(s)) + H(s,x(s)) + L^0(s,x_k(s)) \\
        \displaystyle \qquad \qquad \quad + \sum^m_{i=1} u^i(s) L^i(s,x_k(s)) + \bigg( \frac{\partial L^0}{\partial x}(s,x_k(s)) \\
        \displaystyle \qquad \qquad \quad + \sum^m_{i=1} u^i_k(s) \frac{\partial L^i}{\partial x}(s,x_k(s)) \bigg) ( x(s) - x_k(s) ) \bigg) \; \mathrm{d}s \\
        \dot{x}(s) = f_{k+1}(s,x(s),u(s)) , \quad x(0) = x^0 \\
        \displaystyle \quad \triangleq f_0(s,x_k(s)) + \sum^m_{i=1} u^i(s) f_i(s,x_k(s)) + \bigg( \frac{\partial f_0}{\partial x}(s,x_k(s)) \\
        \displaystyle \qquad + \sum^m_{i=1} u^i_k(s) \frac{\partial f_i}{\partial x}(s,x_k(s)) \bigg) ( x(s) - x_k(s) ) , \ \textnormal{a.e.} \ \revThird{s \in [0,t_f]} \\
        \displaystyle \revThird{g_{k+1}(x(t_f)) \triangleq g(x_k(t_f)) + \frac{\partial g}{\partial x}(x_k(t_f)) ( x(t_f) - x_k(t_f) ) = 0} \\
        \displaystyle \revThird{\int^{t_f}_0 \| x(s) - x_k(s) \|^2 \; \mathrm{d}s \le \Delta_{k+1}}
    \end{cases}
\end{equation*}}
where all the \textit{non-convex contributions} of OCP have been linearized around \revThird{$(x_k,u_k)$}, which for $k \geq 1$ is a solution to the subproblem LOCP$^{\Delta}_k$ at the previous iteration. Accordingly, \revThird{$(x_{k+1},u_{k+1})$} always denotes a solution to the subproblem LOCP$^{\Delta}_{k+1}$. Each subproblem LOCP$^{\Delta}_k$ is convex in the sense that after a discretization in time through any time-linear integration scheme (e.g., Euler schemes, trapezoidal rule, etc.), we end up with a finite-dimensional convex program that can be solved numerically via convex optimization methods. In particular, linearizations of $G$ and $H$ are not required since these mappings are already convex. Finally, we have introduced convex \textit{trust-region constraints}
\begin{equation} \label{eq_trust}
    \revThird{\int^{t_f}_0 \| x(s) - x_k(s) \|^2 \; \mathrm{d}s \le \Delta_{k+1}} .
\end{equation}
These are crucial to guiding the convergence of SCP in the presence of linearization errors. Since the control variable already appears linearly within the non-convex quantities defining OCP, trust-region constraints are not needed for control. We remark that although it might seem more natural to impose pointwise trust-region constraints at each time \revThird{$s \in [0,t_f]$}, the $L^2$-type constraints \eqref{eq_trust} are sufficient to perform a convergence analysis, and importantly, they are less restrictive. The \textit{trust-region radii} $(\Delta_k)_{k \in \mathbb{N}} \subseteq \mathbb{R}_+$ represent optimization parameters and may be updated through iterations to improve the search for a solution at each next iteration. Effective choices of such an updating rule will be discussed in the next section.

The definition of every convex subproblem by induction makes sense only if we can claim that: (1) at each step, the optimal trajectory $x_k$ is defined in the entire interval \revThird{$[0,t_f]$}, and (2) there exists (at least one) optimal solution at each step. The answer to the first question is contained in the following lemma, whose proof relies on routine application of the Gr\"onwall inequality and is postponed in the Appendix.
\begin{lmm}[Boundness of trajectories] \label{lemma_boundness}
Let $\textnormal{supp}f_i$ denote the support of $f_i$, $i=0,\dots,m$. If $U$ and $\textnormal{supp}f_i$, $i=0,\dots,m$ are compact, then each $x_k$ is defined in the entire interval \revThird{$[0,t_f]$} and uniformly bounded for every \revThird{$t \in [0,t_f]$} and $k \in \mathbb{N}$.
\end{lmm}

To answer the second question, we should provide sufficient conditions under which LOCP$^{\Delta}_{k+1}$ admits a solution for each $k \in \mathbb{N}$. To this purpose, we assume the following:
\begin{Itemize}
    \item[$(A_1)$] For every $k \in \mathbb{N}$, the subproblem LOCP$^{\Delta}_{k+1}$ is feasible.
\end{Itemize}
As a classical result, under $(A_1)$, for every $k \in \mathbb{N}$, the subproblem LOCP$^{\Delta}_{k+1}$ has an optimal solution \revThird{$(x_{k+1},u_{k+1})$}, which makes the above definition of each convex subproblem by induction well-posed (see, e.g., \cite{LeeMarkus1967}).

\begin{rmk}
    In practical contexts, $(A_1)$ is often satisfied. 
    This assumption is well-motivated, because, up to a slight modification, each subproblem LOCP$^{\Delta}_k$ is \textit{generically} feasible in the following sense. In the presence of \eqref{eq_trust}, the feasibility of each subproblem would be a consequence of the controllability of its linear dynamics, which is in turn equivalent to the invertibility of its Gramian matrix (see, e.g., \cite{LeeMarkus1967}; the constraints \eqref{eq_trust} force any admissible trajectory of LOCP$^{\Delta}_{k+1}$ to lie within a tubular neighborhood around $x_k$, thus, as a classical result, the controllability criterion in \cite{LeeMarkus1967} applies by restriction to this tubular neighborhood). 
    \revThird{Since the subset of invertible matrices is dense, Gramian matrices are \textit{almost always} (i.e., in a topological sense) invertible.} Linearized dynamics are thus almost always controllable, which implies that each subproblem is feasible. As an important remark, feasibility is preserved through time discretization, making any time-discretized version of the convex subproblems well-posed numerically. Indeed, time discretization maps the continuous linear dynamics into a system of linear equations. Since the set of full-rank matrices is also dense, similar reasoning shows that the discretized subproblems are also almost always feasible. In conclusion, $(A_1)$ is a mild and well-justified assumption.
\end{rmk}


\subsection{Algorithmic Framework} \label{sec_Algo}

The objective of our SCP formulation can be stated as follows: to find locally-optimal solutions to OCP by iteratively solving each subproblem LOCP$^{\Delta}_k$ until the sequence \revThird{$(\Delta_k,x_k,u_k)_{k \in \mathbb{N}}$, where $(x_k,u_k)$} is a solution to LOCP$^{\Delta}_k$, satisfies some \textit{convergence criterion} (to be defined later, \revThird{see Section \ref{sec_numerics}}). We propose pursuing this objective by adopting (pseudo-) Algorithm \ref{algo_SCP}, which is designed to return a locally-optimal solution to OCP, up to small approximation errors.

\SetKwInOut{Input}{Input}
\SetKwInOut{Output}{Output}
\SetKwInOut{Data}{Data}

\begin{algorithm}
\caption{Sequential Convex Programming} \label{algo_SCP}
\Input{Guess trajectory $x_0$ and control $u_0$.}
\Output{Solution \revThird{$(x_k,u_k)$} to LOCP$^{\Delta}_k$ for some $k$.}
\Data{Initial trust-region radius $\Delta_0 > 0$.}

\Begin
{
    $k = 0$, $\Delta_{k+1} = \Delta_k$\\
    \While{$(u_k)_{k \in \mathbb{N}}$ \textnormal{has not converged}}
    {
        Solve LOCP$^{\Delta}_{k+1}$ for \revThird{$(x_{k+1},u_{k+1})$}\\
        
        \revThird{$\Delta_{k+1} = \textbf{UpdateRule}(x_{k+1},u_{k+1},x_k,u_k)$} 
        
            
        
            

        $k \gets k+1$\\

    }
    
    \Return{\revThird{$(x_{k-1},u_{k-1})$}}
}
\end{algorithm}

Algorithm \ref{algo_SCP} requires the user to provide a rule \textbf{UpdateRule} to update the values of the trust-region radius. This rule should primarily aim to prevent accepting solutions at each iteration that are misguided by significant linearization error. A priori, we only require that \textbf{UpdateRule} is such that the sequence of trust-region radii $(\Delta_k)_{k \in \mathbb{N}}$ converges to zero (in particular, $(\Delta_k)_{k \in \mathbb{N}}$ is bounded). In the next section, we show that this numerical requirement, together with other mild assumptions, are sufficient to establish convergence guarantees for Algorithm \ref{algo_SCP}. An example for \textbf{UpdateRule} will be provided in Section \ref{sec_numerics} when discussing numerical simulations.

The algorithm terminates when the sequence of controls $(u_k)_{k \in \mathbb{N}}$ converges with respect to some user-defined topology (as we will see shortly, convergence is always achieved in some specific sense by at least one subsequence of $(u_k)_{k \in \mathbb{N}}$; in Section \ref{sec_numerics}, we \revThird{propose an approximate stopping criterion to}  check the convergence of the sequence $(u_k)_{k \in \mathbb{N}}$). Whenever such convergence is achieved (in some specific sense; see the next section), we may claim Algorithm \ref{algo_SCP} has found a candidate locally-optimal solution for OCP (see Theorem \ref{theo_main} in the next section). The reason that only the convergence of the sequence of controls suffices to claim success is contained in our convergence result (see Theorem \ref{theo_main} in the next section). To measure the convergence of $(u_k)_{k \in \mathbb{N}}$, some topologies 
are better than others, and in particular, under mild assumptions one can prove that, up to some subsequence, $(u_k)_{k \in \mathbb{N}}$ always converges with respect to the weak topology of $L^2$. In turn, this may be interpreted as a result of weak existence of successful trajectories for Algorithm \ref{algo_SCP} when selecting the $L^2$-weak topology as convergence metric. In practice, Algorithm \ref{algo_SCP} is numerically applied to time-discretized versions of each subproblem LOCP$^{\Delta}_k$. Thus we will show that our conclusions regarding convergence behavior still hold in a discrete context, up to discretization errors (see the next section).

\subsection{Convergence Analysis} \label{sec_guarantees}

We now turn to the convergence of Algorithm \ref{algo_SCP}. Under mild assumptions, our analysis provides three key results:
\begin{enumerate}
    \item[R1] When the sequence $(u_k)_{k \in \mathbb{N}}$ returned by Algorithm \ref{algo_SCP} converges, the limit is a stationary point for OCP in the sense of the \textit{Pontryagin Maximum Principle} (PMP).
    \item[R2] There always exists a subsequence of $(u_k)_{k \in \mathbb{N}}$ that converges to a stationary point of OCP for the weak topology of $L^2$.
    \item[R3] This converging behavior transfers to time-discretization of Algorithm \ref{algo_SCP}, i.e., versions for which we adopt time-discretization of subproblems LOCP$^{\Delta}_k$.
\end{enumerate}
Result R1 is the core of our analysis and roughly states that whenever Algorithm \ref{algo_SCP} achieves convergence, a candidate locally-optimal solution to the original problem has been found. For the proof of this result, we build upon the PMP.

Before focusing on the convergence result, we recall the statement of the PMP and list our main assumptions. \revThird{For the sake of clarity, we introduce the PMP related to OCP and the PMP related to each convexified problem LOCP$^{\Delta}_k$ separately}.

\revThird{\textit{PMP related to OCP:}} For every $p \in \mathbb{R}^n$ and $p^0 \in \mathbb{R}$, define the \textit{Hamiltonian}
$$
H(s,x,p,p^0,u) = p^{\top} f(s,x,u) + p^0 f^0(s,x,u) .
$$

\begin{thm}[\revThird{PMP for OCP \cite{Pontryagin1987}}] \label{theo_PMP}
Let \revThird{$(x^*,u^*)$} be a locally-optimal solution to OCP. There exist an absolutely-continuous function \revThird{$p : [0,t_f] \rightarrow \mathbb{R}^n$} and a constant $p^0 \le 0$, such that the following hold:
\begin{itemize}
    \item Non-Triviality Condition: $(p,p^0) \neq 0$.
    \item Adjoint Equation: Almost everywhere in \revThird{$[0,t_f]$},
    $$
    \dot{p}(s) = -\frac{\partial H}{\partial x}(s,x^*(s),p(s),p^0,u^*(s)) .
    $$
    \item Maximality Condition: Almost everywhere in \revThird{$[0,t_f]$},
    \begin{align*}
        H(s,x^*(s),p(s),p^0, \ &u^*(s)) = \\
        &= \underset{v \in U}{\max} \ H(s,x^*(s),p(s),p^0,v) .
    \end{align*}
    \item Transversality Condition: It holds that
    $$
    \revThird{p(t_f) \ \perp \ \textnormal{ker} \frac{\partial g}{\partial x}(x^*(t_f))} .
    $$
\end{itemize}
\end{thm}

\vspace{5pt}

\noindent \revThird{A tuple $(x^*,p,p^0,u^*)$ satisfying Theorem \ref{theo_PMP} is called (Pontryagin) extremal for OCP. Note that, thanks to the equivalence $(\textnormal{ker} \ A)^{\perp} = \textnormal{Im} \ A^{\top}$ for a matrix $A$, the transversality condition entails that $p(t_f) = \sum^{\ell_g}_{i=1} \lambda_i \nabla g_i(x^*(t_f))$, for some $\lambda \in \mathbb{R}^{\ell_g}$.}

\revThird{\textit{PMP related to LOCP$^{\Delta}_k$:} For every $k \ge 1$, $p \in \mathbb{R}^n$, and $p^0 , p^1 \in \mathbb{R}$, define the \textit{Hamiltonian}}
\begin{align*}
    \revThird{H_k(s,x,p,p^0,p^1,u)} &\revThird{= p^{\top} f_k(s,x,u)} \\
    &\revThird{+ p^0 f^0_k(s,x,u) + p^1 \| x - x_{k-1}(s) \|^2 .}
\end{align*}

\revThird{\begin{thm}[Weak PMP for LOCP$^{\Delta}_k$ \cite{Hartl1995}] \label{theo_PMPk}
Let $k \ge 1$ and $(x_k,u_k)$ be a locally-optimal solution to LOCP$^{\Delta}_k$. There exist an absolutely-continuous function $p_k : [0,t_f] \rightarrow \mathbb{R}^n$ and two constants $p^0_k \le 0$, $p^1_k \in \mathbb{R}$, such that the following hold:
\begin{itemize}
    \item Non-Triviality Condition: $(p_k,p^0_k,p^1_k) \neq 0$.
    \item Adjoint Equation: Almost everywhere in \revThird{$[0,t_f]$},
    \begin{align*}
        \dot{p}_k(s) = -\frac{\partial H_k}{\partial x}(s,x_k(s),p_k(s),p^0_k,p^1_k,u_k(s)) .
    \end{align*}
    \item Maximality Condition: Almost everywhere in \revThird{$[0,t_f]$},
    \begin{align*}
        H_k(s, \; &x_k(s),p_k(s),p^0_k,p^1_k,u_k(s)) = \\
        &\underset{v \in U}{\max} \ H(s,x_k(s),p_k(s),p^0_k,p^1_k,v) .
    \end{align*}
    \item Transversality Condition: It holds that
    $$
    \revThird{p_k(t_f) \ \perp \ \textnormal{ker} \frac{\partial g}{\partial x}(x_k(t_f))} .
    $$
\end{itemize}
\end{thm}}

\vspace{5pt}

\noindent \revThird{A tuple $(x_k,p_k,p^0_k,p^1_k,u_k)$ satisfying Theorem \ref{theo_PMPk} is called extremal for LOCP$^{\Delta}_k$.}

\revThird{\begin{rmk} \label{remark_slack}
    Theorems \ref{theo_PMP} and \ref{theo_PMPk} provide first-order necessary conditions for optimality, thus extremals are candidate local optima. It is worth noting that, thanks to the new variable
    \begin{equation} \label{eq_changeVar}
        y(t) \triangleq \int^t_0 \| x(s) - x_k(s) \|^2 \; \mathrm{d}s ,
    \end{equation}
    constraints \eqref{eq_trust} may be written as $y(t_f) - \Delta_{k+1} \le 0$. By leveraging this transformation, LOCP$^{\Delta}_{k+1}$ may be reformulated as an optimal control problem with final inequality constraints but without state constraints. Thus, the multipliers introduced in Theorems \ref{theo_PMP} and \ref{theo_PMPk} are continuous functions of time (see our proof in Section \ref{sec_proof}; compare also with \cite[Theorem 4.1]{Hartl1995} under final inequality constraints only). Finally, although the conditions listed in Theorem \ref{theo_PMP} are essentially sharp, the statement of Theorem \ref{theo_PMPk} may be strengthened as follows. If $(x_{k+1},p_{k+1},p^0_{k+1},p^1_{k+1},u_{k+1})$ is an extremal for LOCP$^{\Delta}_{k+1}$, one can additionally prove that (see \cite{Hartl1995}; note that \cite{Hartl1995} considers maximization of rewards rather than minimization of costs, thus multipliers must change sign) $p^1_{k+1} \le 0$ and
    $$
    p^1_{k+1} \left( \int^{t_f}_0 \| x_{k+1}(s) - x_k(s) \|^2 \; \mathrm{d}s - \Delta_{k+1} \right) = 0
    $$
    (the latter is know as \textit{slack condition}), which motivates the choice ``weak PMP for LOCP$^{\Delta}_{k+1}$'' as name for Theorem \ref{theo_PMPk}. Nevertheless, since the constraints \eqref{eq_trust} do not appear in the original problem OCP, we do not need to leverage these latter additional conditions on $p^1_{k+1}$, i.e., Theorem \ref{theo_PMPk} suffices to establish convergence for SCP when applied to solve OCP. 
\end{rmk}}

\revThird{Assumption $(A_1)$ suffices to obtain the result R1} 
(see Theorem \ref{theo_main} below). To prove 
result R2, additional regularity \revThird{on the data defining OCP} is required. Specifically, we introduce the following technical condition:
\revFourth{
\begin{Itemize}
    \item[$(A_2)$] The mapping $G : \mathbb{R}^{m+1} \rightarrow \mathbb{R}$ is $u$-strongly convex, i.e., there exists $\sigma > 0$ such that for every $s \in \mathbb{R}$ and every $u_1 , u_2 \in \mathbb{R}^m$,
    \begin{align*}
        G(s,\lambda u_1 + (1 - \lambda) &u_2) + \frac{1}{2} \sigma \lambda (1 - \lambda) \| u_1 - u_2 \|^2 \le \\
        &\le \lambda G(s,u_1) + (1 - \lambda) G(s,u_2) .
    \end{align*}
\end{Itemize}}

Our main convergence result reads as follows,
\revThird{\begin{thm}[Guarantees of convergence for SCP] \label{theo_main}
Assume that $(A_1)$ holds and that Algorithm \ref{algo_SCP} returns a sequence $(\Delta_k,u_k,x_k)_{k \in \mathbb{N}}$ such that $\Delta_k \rightarrow 0$ and, for every $k \in \mathbb{N}$, the tuple $(u_{k+1},x_{k+1})$ locally solves LOCP$^{\Delta}_{k+1}$.
\begin{enumerate}
    \item Assume that the sequence of controls $(u_k)_{k \in \mathbb{N}}$ converges to some $u^* \in \mathcal{U}$ for the strong topology of $L^2$. Let $x^* : [0,t_f] \rightarrow \mathbb{R}^n$ denote the solution to the dynamics of OCP associated with the control $u^*$. Then, there exist a sequence $(p_k,p^0_k,p^1_k)_{k \in \mathbb{N}}$ and a tuple $(p,p^0)$, with $p : [0,t_f] \rightarrow \mathbb{R}^n$ absolutely continuous, $p^0 \le 0$, such that:
    \begin{enumerate}
        \item $(x_k,p_k,p^0_k,p^1_k,u_k)$ is an extremal for LOCP$^{\Delta}_k$ and these convergence results hold:
        \begin{itemize}
            \item $x_k \rightarrow x^*$ for the strong topology of $C^0$.
            \item Up to some subsequence, $p_k \rightarrow p$ for the strong topology of $C^0$, and $p^0_k \rightarrow p^0$.
        \end{itemize}
        \item If $(p(t_f),p^0) \neq 0$, i.e., the value to which $(p_k(t_f),p^0_k)_{k \in \mathbb{N}}$ converges is not zero, then $(x^*,p,p^0,u^*)$ is an extremal for OCP.
    \end{enumerate}
    \item Assume that $(A_2)$ holds and the sequence of controls $(u_k)_{k \in \mathbb{N}}$ converges to $u^* \in \mathcal{U}$ for the weak topology of $L^2$. If $(p_k,p^0_k,p^1_k)_{k \in \mathbb{N}}$ is such that $p^0_k \neq 0$ for every $k \ge 1$, then the statements in 1.a-1.b above remains true. In addition, there always exists a subsequence $(u_{k_j})_{j \in \mathbb{N}}$ that converges to some $u^* \in \mathcal{U}$ for the weak topology of $L^2$, such that the statements in 1.a-1.b above are true.
\end{enumerate}
\end{thm}}

The guarantees offered by Theorem \ref{theo_main} read as follows. \revThird{Under $(A_1)$} and by selecting a \textit{shrinking-to-zero} sequence of trust-region radii, if iteratively solving problems LOCP$^{\Delta}_k$ returns a sequence of \revThird{extremals $(x_k,p_k,p^0_k,p^1_k,u_k)_{k \in \mathbb{N}}$ such that (1) $(u_k)_{k \in \mathbb{N}}$ converges with respect to the strong topology of $L^2$, and (2) $(p_k(t_f),p^0_k) \rightarrow \ell \neq 0$,} then there exists a Pontryagin extremal for the original problem, i.e., a candidate (local) solution to OCP \revThird{to which $(x_k,p_k,p^0_k,u_k)_{k \in \mathbb{N}}$ converges}
, which formalizes result R1. Moreover, under \revThird{additional regularity on the data defining OCP and} the additional assumption that the generated sequence of \revThird{extremals $(x_k,p_k,p^0_k,p^1_k,u_k)_{k \in \mathbb{N}}$ is such that $p^0_k \neq 0$ for every $k \ge 1$, a converging sequence of controls $(u_k)_{k \in \mathbb{N}}$} always exists
, which formalizes result R2. This can be clearly interpreted as a ``weak'' guarantee of success for SCP, where ``weak'' refers to the fact that only a subsequence of \revThird{$(u_k)_{k \in \mathbb{N}}$} converges, \revThird{a guarantee which is often sought and leveraged from the optimization community, see, e.g., \cite[Theorem 3.4]{Lu2013}, in which ``accumulation points'' are considered}.

\begin{rmk}
    When SCP achieves convergence, \revThird{$(A_2)$ and the requirement $p^0_k \neq 0$ for every $k \ge 1$ are} not needed for the derivation of theoretical guarantees on local optimality. \revThird{On the other hand, the requirements $(p_k(t_f),p^0_k) \rightarrow \ell \neq 0$ (or equivalently $(p(t_f),p^0) \neq 0$) and $p^0_k \neq 0$ for every $k \ge 1$ play the role of ``qualification conditions'' (compare also with condition (14) in \cite[Theorem 3.4]{Lu2013}), a standard requirement in optimization which can be easily checked numerically, see Section \ref{sec_numerics}. In particular, the requirement $p^0_k \neq 0$ for every $k \ge 1$ means that each extremal $(x_k,p_k,p^0_k,p^1_k,u_k)_{k \in \mathbb{N}}$ is \textit{normal} (see \cite{ShvartsmanVinter2006} for a definition), and normality of extremals naturally occurs in many optimal control problem settings (see \cite[Corollary 2.9]{ChitourJeanEtAl2008}, in which the authors show that normality of extremals holds \textit{generically} true as long as optimal controls take value in the interior of the control domain; see also \cite{Bettiol2007,Vinter2010} for additional general settings where extremals are normal), further justifying the requirement $p^0_k \neq 0$ for $k \ge 1$.}
\end{rmk}

\begin{rmk}
    Those guarantees adapt when time discretization is adopted to numerically solve each convex subproblem, which is the most frequently used and reliable technique in practice. To see this, fix a time-discretization scheme and consider the discretized version of OCP. Any candidate locally-optimal solution to this discrete formulation satisfies the Karush-Kuhn-Tucker (KKT) conditions. \textit{If} a Pontryagin extremal of OCP exists, the limit of points satisfying the KKT for the discretized version of OCP as the time step tends to zero converges to the aforementioned Pontryagin extremal of OCP (more precisely, up to some subsequence; the reader can find more details in \cite{GollmannKernEtAl2009}). Theorem \ref{theo_main} exactly provides conditions under which the ``if sentence'' above holds true, that is, conditions under which the aforementioned Pontryagin extremal of OCP exists, thus endowing Algorithm \ref{algo_SCP} with correctness guarantees that are independent of any time discretization the user may select (Euler, Runge-Kutta, etc.).
\end{rmk}

\subsection{Proof of the Convergence Result} \label{sec_proof}

We split the proof of Theorem \ref{theo_main} in three main steps. First, we retrace the main steps of the proof of the PMP to introduce necessary notation and expressions. Second, we show the convergence of trajectories and controls, together with the convergence of \revThird{Pontryagin variations (see the paragraph below for a definition)}. The latter represents the cornerstone of the proof and paves the way for the final step, which consists of proving the convergence of the Pontryagin extremals. 

\subsubsection{Pontryagin Variations}

\revThird{Let $u \in \mathcal{U}$ be a feasible control for OCP, with associated trajectory $x_u$ in $[0,t_f]$. For every $r \in [0,t_f]$ Lebesgue point of $u$, and $v \in U$, we define}
\begingroup
\begin{equation} \label{ref:localVarOCP}
\revThird{\tilde \xi^{r,v}_u \triangleq \left( \begin{array}{c}
\displaystyle f^0(r,x_u(r),v) - f^0(r,x_u(r),u(r)) \\
\displaystyle f(r,x_u(r),v) - f(r,x_u(r),u(r))
\end{array} \right) \in \mathbb{R}^{n+1} .}
\end{equation}
\endgroup
The \textit{variation trajectory} \revThird{$\tilde z^{r,v}_u : [0,t_f] \rightarrow \mathbb{R}^{n+1}$ related to $r \in [0,t_f]$}, to $v \in U$, and to the feasible control $u \in \mathcal{U}$ for OCP is defined to be the unique \revThird{(global in $[0,t_f]$)} solution to the following system of linear differential equations
\begin{equation} \label{ref:varOCP}
\begin{cases}
\revThird{\dot{\tilde z}(s) = \displaystyle \left( \begin{array}{cc}
\displaystyle \frac{\partial f^0}{\partial x}(s,x_u(s),u(s)) & 0 \\[-7.5pt]
& \vdots \\[-7pt]
\displaystyle \frac{\partial f}{\partial x}(s,x_u(s),u(s)) & 0
\end{array} \right) \tilde z(s)} \\
\tilde z(r) = \tilde \xi^{r,v}_u .
\end{cases}
\end{equation}
The proof of the PMP goes by contradiction, considering \textit{Pontryagin variations} (see, e.g., \cite{AgrachevSachkov2004}). We define those to be all the vectors \revThird{$\tilde z^{r,v}_u(t_f)$, where $r \in (0,t_f)$} is a Lebesgue point of $u$ and $v \in U$. In particular, if \revThird{$(x_u,u)$} is locally optimal for OCP, then one infers the existence of a nontrivial tuple \revThird{$(\mathfrak{p},p^0) \in \mathbb{R}^{\ell_g+1}$ ($\mathfrak{p}$ is a row vector)}, with $p^0 \le 0$, satisfying, for \revThird{all $r \in (0,t_f)$} Lebesgue point of $u$ and all $v \in U$,
\begin{equation} \label{ref:variationalOCP}
    \revThird{\displaystyle \left( p^0 , \mathfrak{p} \frac{\partial g}{\partial x}(x_u(t_f)) \right) \cdot \tilde z^{r,v}_u(t_f) \le 0 .}
\end{equation}
The non-triviality condition, the adjoint equation, the maximality condition, and the transversality condition listed in Theorem \ref{theo_PMP} derive from \eqref{ref:variationalOCP}. Specifically, it can be shown that a tuple \revThird{$(x_u,p,p^0,u)$} is a Pontryagin extremal for OCP if and only if the nontrivial tuple \revThird{$\displaystyle \left(p(t_f) = \mathfrak{p} \frac{\partial g}{\partial x}(x_u(t_f)),p^0\right) \in \mathbb{R}^{n+1}$} with $p^0 \le 0$ satisfies \eqref{ref:variationalOCP} (see, e.g., \cite{AgrachevSachkov2004}). For this reason, \revThird{$(x_u,\mathfrak{p},p^0,u)$} is also called extremal for OCP.


\revThird{Let us show how, thanks to the change of variable \eqref{eq_changeVar}, the previous conclusions adapt to each subproblem built in Algorithm \ref{algo_SCP}.} Specifically, for every $k \in \mathbb{N}$, \revThird{assuming $(A_1)$,} let \revThird{$(x_{k+1},u_{k+1})$} denote a solution to \revThird{LOCP$^{\Delta}_{k+1}$}, with related trust-region radius $\Delta_{k+1}$, \revThird{and introduce the smooth curve}
$$
\revThird{y_{k+1} : [0,t_f] \rightarrow \mathbb{R} : t \mapsto \int^t_0 \| x_{k+1}(s) - x_k(s) \|^2 \; \mathrm{d}s .}
$$
\revThird{Clearly, condition \eqref{eq_trust} is equivalent to $y_{k+1}(t_f) - \Delta_{k+1} \le 0$. Next, consider the extended smooth dynamics}
$$
\revThird{\tilde f_{k+1}(s,x,u) \triangleq \left(\begin{array}{c}
    f(s,x,u)\\
    \| x - x_k(s) \|^2
\end{array}\right)}
$$
\revThird{and for every $r \in [0,t_f]$} Lebesgue point of $u_{k+1}$ and every $v \in U$ define
{\small
\begin{equation} \label{ref:localVarLOCP}
\revThird{\tilde \xi^{r,v}_{k+1} = \left( \begin{array}{c}
\displaystyle f^0_{k+1}(r,x_{k+1}(r),v) - f^0_{k+1}(r,x_{k+1}(r),u_{k+1}(r)) \\
\displaystyle \tilde f_{k+1}(r,x_{k+1}(r),v) - \tilde f_{k+1}(r,x_{k+1}(r),u_{k+1}(r))
\end{array} \right) .}
\end{equation}}
Straightforward computations show that the control $u_k$ does not explicitly appear within expression \eqref{ref:localVarLOCP}. Thus the time $r \in [0,t^{k+1}_f]$ needs to be a Lebesgue point of $u_{k+1}$ only. We define the variation trajectory \revThird{$\tilde z^{r,v}_{k+1} : [0,t_f] \rightarrow \mathbb{R}^{n+2}$ related to $r \in [0,t_f]$}, to $v \in U$, and to the locally-optimal control $u_{k+1}$ for \revThird{LOCP$^{\Delta}_{k+1}$} to be the unique \revThird{(global in $[0,t_f]$)} solution to the following system of linear differential equations
\begin{eqnarray} \label{ref:varLOCP}
\begin{cases}
\revThird{\dot{\tilde z}(s) = \displaystyle \left( \begin{array}{ccc}
\displaystyle \frac{\partial f^0_{k+1}}{\partial x}(s,x_{k+1}(s),u_{k+1}(s)) & 0 & 0 \\[-7.5pt]
& \vdots & \vdots \\[-7pt]
\displaystyle \frac{\partial \tilde f_{k+1}}{\partial x}(s,x_{k+1}(s),u_{k+1}(s)) & 0 & 0
\end{array} \right) \tilde z(s)} \\
\tilde z(r) = \tilde \xi^{r,v}_{k+1} .
\end{cases}
\end{eqnarray}
The Pontryagin variations related to \revThird{LOCP$^{\Delta}_{k+1}$} are all the vectors \revThird{$\tilde z^{r,v}_{k+1}(t_f)$}, where \revThird{$r \in (0,t_f)$} is a Lebesgue point of $u_{k+1}$ and $v \in U$. \revThird{At this step, one may easily extend the proof of \cite[Theorem 12.13]{AgrachevSachkov2004} to the case of augmented final constraints $g(x_{k+1}(t_f)) = 0$ and $y_{k+1}(t_f) - \Delta_{k+1} \le 0$, and from the local optimality of $(x_{k+1},u_{k+1})$ for LOCP$^{\Delta}_{k+1}$, infer the existence of a nontrivial tuple $(\mathfrak{p}_{k+1},p^0_{k+1},p^1_{k+1}) \in \mathbb{R}^{\ell_g+2}$ ($\mathfrak{p}_{k+1}$ is a row vector), with $p^0_{k+1} \le 0$, satisfying, for $r \in (0,t_f)$ (Lebesgue for $u_{k+1}$) and $v \in U$,}
\begin{equation} \label{ref:variationalLOCP}
    \revThird{\displaystyle \left( p^0_{k+1} , \mathfrak{p}_{k+1} \frac{\partial g}{\partial x}(x_{k+1}(t_f)) , p^1_{k+1} \right) \cdot \tilde z^{r,v}_{k+1}(t_f) \le 0.}
\end{equation}
\revThird{The non-triviality condition, the adjoint equation, the maximality condition, and the transversality condition listed in Theorem \ref{theo_PMPk} derive from algebraic manipulations on \eqref{ref:variationalLOCP}. Again, we stress the fact that the necessary conditions for optimality offered by \eqref{ref:variationalLOCP} are not exhaustive, in that the sign of the multiplier $p^1_{k+1} \in \mathbb{R}$ and additional slack conditions may be characterized as we mentioned in Remark \ref{remark_slack}. Nevertheless, as we will show shortly \eqref{ref:variationalLOCP} suffices to prove Theorem \ref{theo_main}.}

The main step in the proof of Theorem \ref{theo_main} consists of showing that it is possible to pass the limit $k \rightarrow \infty$ inside \eqref{ref:variationalLOCP}, recovering a nontrivial tuple $(\mathfrak{p},p^0) \in \mathbb{R}^{\ell_g+1}$ with $p^0 \le 0$ that satisfies \eqref{ref:variationalOCP}. Due to the equivalence between the conditions of the PMP and \eqref{ref:variationalOCP}, this is sufficient to prove the existence of a Pontryagin extremal for OCP. We will show that this also implies the convergences stated in Theorem \ref{theo_main}. We will only focus on proving the last part of 
2) in Theorem \ref{theo_main}, by adopting the additional assumption \revThird{$(A_2)$ and the requirement $p^0_k \neq 0$ for every $k \ge 1$}, since proofs of the remaining cases are similar and easier to construct.

\subsubsection{Convergence of Controls and Trajectories}

\revThird{By the compactness of $U$}, the sequence \revThird{$(u_k)_{k \in \mathbb{N}} \subseteq L^2([0,t_f];U)$} is uniformly bounded in \revThird{$L^2([0,t_f];\mathbb{R}^m)$}. Since \revThird{$L^2([0,t_f];U)$} is closed and convex in \revThird{$L^2([0,t_f];\mathbb{R}^m)$} (because $U$ is compact and convex) and \revThird{$L^2([0,t_f];\mathbb{R}^m)$} is reflexive, there exists a control \revThird{$u^* \in L^2([0,t_f];U)$} (in particular $u^* \in \mathcal{U}$) such that we can extract a subsequence (still denoted $(u_k)_{k \in \mathbb{N}}$) that converges to $u^*$ for the weak topology of $L^2$. We denote by $x^*$ the trajectory solution to the dynamics of OCP related to $u^*$, \revThird{which is defined on $[0,t_f]$ thanks to Lemma \ref{lemma_boundness}.}

Next, recalling that thanks to Lemma \ref{lemma_boundness} the trajectories $x_k$ are defined in \revThird{$[0,t_f]$} and uniformly bounded, we show that
\begin{equation} \label{eq_convTraj}
    \revThird{\underset{s \in [0,t_f]}{\sup} \ \| x_k(s) - x^*(s) \| \longrightarrow 0}
\end{equation}
for $k \rightarrow \infty$. This will provide the desired convergence of trajectories. For \revThird{$t \in [0,t_f]$} we have that
{\footnotesize
\begin{align*}
    &\| x_{k+1}(t) - x^*(t) \| \le \int^t_0 \left\| f_0(s,x_k(s)) - f_0(s,x^*(s)) \right\| \; \mathrm{d}s \\
    &\quad + \sum^m_{i=1} \left\| \int^t_0 \left( u^i_{k+1}(s) f_i(s,x_k(s)) - u^i(s) f_i(s,x^*(s)) \right) \; \mathrm{d}s \right\| \\
    &\quad + \int^t_0 \left\| \frac{\partial f_0}{\partial x}(s,x_k(s)) \right\| \| x_{k+1}(s) - x_k(s) \| \; \mathrm{d}s \\
    &\quad + \sum^m_{i=1} \int^t_0 \left\| u^i_k(s) \frac{\partial f_i}{\partial x}(s,x_k(s)) \right\| \| x_{k+1}(s) - x_k(s) \| \; \mathrm{d}s \\
    &\le C_1 \bigg( \int^t_0 \left\| x_{k+1}(s) - x^*(s) \right\| \; \mathrm{d}s + \revThird{\Delta^{\frac{1}{2}}_{k+1}} \\
    &\quad + \underbrace{\sum^m_{i=1} \left\| \int^t_0 f_i(s,x^*(s)) \left( u^i_{k+1}(s) - (u^*)^i(s) \right) \; \mathrm{d}s \right\|}_{\triangleq \delta^{u,1}_{k+1}(t)} \bigg)
\end{align*}}
where $C_1 \ge 0$ is a constant that stems from the uniform boundedness of $(x_k)_{k \in \mathbb{N}}$ (see also the proof of Lemma \ref{lemma_boundness} in the Appendix). Now, the definition of weak convergence in $L^2$ gives that, for every fixed \revThird{$t \in [0,t_f]$}, $\delta^{u,1}_{k+1}(t) \rightarrow 0$ for $k \rightarrow \infty$. In addition, by the compactness of $U$ and $\textnormal{supp} f_i$, there exists a constant $C_2 \ge 0$ such that, for every \revThird{$t, s \in [0,t_f]$}
$$
| \delta^{u,1}_{k+1}(t) - \delta^{u,1}_{k+1}(s) | \le C_2 | t - s |
$$
uniformly with respect to $k \in \mathbb{N}$. Thus, by \cite[Lemma 3.4]{trelat2000}, $\delta^{u,1}_{k+1}(t) \rightarrow 0$ for $k \rightarrow \infty$ uniformly in the interval \revThird{$[0,t_f]$}. \revThird{We conclude thanks to $\Delta_k \rightarrow 0$ and a routine Gr\"onwall inequality argument} (see also the proof of Lemma \ref{lemma_boundness} in the Appendix).

Let us prove that the trajectory \revThird{$x^* : [0,t_f] \rightarrow \mathbb{R}^n$} is feasible for OCP. To do so, we only need to prove that \revThird{$g(x^*(t_f)) = 0$}. Note that from \eqref{eq_convTraj},  $\Delta_k \rightarrow 0$, and the boundedness and convergence of the trajectories, we have that
\begin{align*}
    \revThird{\| g(x^*(t_f)) \|} &\revThird{= \| g(x^*(t_f)) \pm g(x_k(t_f)) \|} \revThird{\le \| g(x_k(t_f)) \|} \\
    &\revThird{+ \underset{x \in K}{\sup} \ \left\| \frac{\partial g}{\partial x}(x) \right\| \| x^*(t_f) - x_k(t_f) \| \longrightarrow 0 ,}
\end{align*}
\revThird{where $K \subseteq \mathbb{R}^n$ is some compact set (see Lemma \ref{lemma_boundness}), from which we conclude that $x^* : [0,t_f] \rightarrow \mathbb{R}^n$} is feasible for OCP.

\subsubsection{Convergence of Pontryagin Variations} \label{sec_variations}

Due to the convergence of controls and trajectories, we can now prove that it is possible to pass the limit $k \rightarrow \infty$ inside \eqref{ref:variationalLOCP}, showing that \eqref{ref:variationalOCP} holds. \revThird{\revFourth{First, thanks to $(A_2)$ and $p^0_k \neq 0$ for every $k \ge 1$, by \cite[Lemma 5.3]{ShvartsmanVinter2006} every control $u_{k+1}$ is continuous.} Therefore, the following result holds (see \cite[Lemma 3.11]{bonalli2019}):}

\begin{lmm}[Pointwise convergence of controls] \label{ref:lemmaLebesgue}
For every \revThird{$r \in (0,t_f)$} Lebesgue point of $u^*$ there exists \revThird{$(r_k)_{k \in \mathbb{N}} \subseteq (0,t_f)$} such that $r_k$ is a Lebesgue point of $u_k$, and the convergences $r_k \rightarrow r$ and $u_k(r_k) \rightarrow u^*(r)$ hold for $k \rightarrow \infty$.
\end{lmm}

Now, fix $r \in (0,t^*_f)$ Lebesgue point of $u^*$, and $v \in U$, and let $(r_k)_{k \in \mathbb{N}}$ be the sequence provided by Lemma \ref{ref:lemmaLebesgue} related to $r$ and $v$. We prove the following convergence:
\begin{equation} \label{eq_convVarTraj}
    \revThird{\underset{s \in [r,t_f]}{\sup} \ \| \tilde z^{r_{k+1},v}_{k+1}(s) - ( \tilde z^{r,v}_{u^*}(s) , 0 ) \| \longrightarrow 0}
\end{equation}
for $k \rightarrow \infty$, where $\tilde z^{r_{k+1},v}_{k+1}$ solves \eqref{ref:varLOCP} with initial condition $\tilde z^{r_{k+1},v}_{k+1}(r_{k+1}) = \tilde \xi^{r_{k+1},v}_{k+1}$ given by \eqref{ref:localVarLOCP}, whereas $\tilde z^{r,v}_{u^*}$ solves \eqref{ref:varOCP} with initial condition $\tilde z^{r,v}_{u^*}(r) = \tilde \xi^{r,v}_{u^*}$ given by \eqref{ref:localVarOCP}. First,
{\footnotesize
\begin{align*}
    &\| \tilde \xi^{r_{k+1},v}_{k+1} - \revThird{( \tilde \xi^{r,v}_{u^*} , 0 )} \| \le \\
    &\le \sum^m_{i=1} |v^i| \| f_i(r_{k+1},x_{k+1}(r_{k+1})) - f_i(r,x^*(r)) \| \\
    &+ \sum^m_{i=1} \| u^i_{k+1}(r_{k+1}) f_i(r_{k+1},x_{k+1}(r_{k+1})) - (u^*)^i(r) f_i(r,x^*(r)) \| \\
    &+ \| G(r_{k+1},v) - G(r,v) \| \\
    &+ \| G(r_{k+1},u_{k+1}(r_{k+1})) - G(r,u^*(r)) \| \\
    &+ \sum^m_{i=1} |v^i| \| L^i(r_{k+1},x_{k+1}(r_{k+1})) - L^i(r,x^*(r)) \| \\
    &+ \sum^m_{i=1} \| u^i_{k+1}(r_{k+1}) L^i(r_{k+1},x_{k+1}(r_{k+1})) - (u^*)^i(r) L^i(r,x^*(r)) \| \\
    &\le C_3 \Big( | r_{k+1} - r | + \| x_{k+1}(r_{k+1}) - x^*(r) \| + \| u_{k+1}(r_{k+1}) - u^*(r) \| \Big)
\end{align*}}
where $C_3 \ge 0$ is a constant, and from Lemma \ref{ref:lemmaLebesgue}
and \eqref{eq_convTraj} we infer that $\| \tilde \xi^{r_{k+1},v}_{k+1} - \revThird{( \tilde \xi^{r,v}_{u^*} , 0 )} \| \rightarrow 0$ for $k \rightarrow \infty$. Second, by leveraging the uniform boundedness of the trajectories, with the same exact argument proposed in the proof of Lemma \ref{lemma_boundness} in the Appendix, one may show that the sequence of variation trajectories $(\tilde z^{r_k,v}_k)_{k \in \mathbb{N}}$ is uniformly bounded in the time interval \revThird{$[r,t_f]$}. From this, for every \revThird{$t \in [r,t_f]$}

{\small
$$
\| \tilde z^{r_{k+1},v}_{k+1}(t) - \revThird{( \tilde z^{r,v}_{u^*}(t) , 0 )} \| \le \| \tilde \xi^{r_{k+1},v}_{k+1} - \revThird{( \tilde \xi^{r,v}_{u^*} , 0 )} \| + C_4 | r_{k+1} - r |
$$
$$
\quad + \int^t_r \bigg( \revThird{2 \| x_{k+1}(s) - x_k(s) \| +} \left\| \frac{\partial f_{k+1}}{\partial x}(s,x_{k+1}(s),u_{k+1}(s)) \right\|
$$
$$
\quad + \left\| \frac{\partial f^0_{k+1}}{\partial x}(s,x_{k+1}(s),u_{k+1}(s)) \right\| \bigg) \| \tilde z^{r_{k+1},v}_{k+1}(s) - \revThird{( \tilde z^{r,v}_{u^*}(s) , 0 )} \| \; \mathrm{d}s
$$
$$\hspace{-7.5ex}+ \bigg\| \int^t_r \bigg( \revThird{\left( \begin{array}{ccc}
    \displaystyle \frac{\partial f^0_{k+1}}{\partial x}(s,x_{k+1}(s),u_{k+1}(s)) & 0 & 0 \\
    \displaystyle \frac{\partial f_{k+1}}{\partial x}(s,x_{k+1}(s),u_{k+1}(s)) & \vdots & \vdots \\
    2 ( x_{k+1}(s) - x_k(s) )^{\top} & 0 & 0
\end{array} \right)}
$$
$$
\qquad \qquad - \revThird{\left( \begin{array}{ccc}
    \displaystyle \frac{\partial f^0}{\partial x}(s,x^*(s),u^*(s)) & 0 & 0 \\
    \displaystyle \frac{\partial f}{\partial x}(s,x^*(s),u^*(s)) & \vdots & \vdots \\
    0 \qquad \qquad \dots & 0 & 0
\end{array} \right) \bigg)} \revThird{\left( \begin{array}{c} \tilde z^{r,v}_{u^*}(s) \\
0 \end{array} \right)} \; \mathrm{d}s \bigg\|
$$
$$
\hspace{-40ex}\le \| \tilde \xi^{r_{k+1},v}_{k+1} - \revThird{( \tilde \xi^{r,v}_{u^*} , 0 )} \|
$$
$$
+ C_4 \bigg( | r_{k+1} - r | + \int^t_r \| \tilde z^{r_{k+1},v}_{k+1}(s) - \revThird{( \tilde z^{r,v}_{u^*}(s) , 0 )} \| \; \mathrm{d}s \revThird{+ \Delta^{\frac{1}{2}}_{k+1}}
$$
$$
\quad + \int^t_r \| x_k(s) - x^*(s) \| \; \mathrm{d}s + \int^t_r \| x_{k+1}(s) - x^*(s) \| \; \mathrm{d}s \bigg)
$$
$$
\quad + \underbrace{\sum^m_{i=1} \left\| \int^t_r F(f_i,L^i,\tilde z^{r,v}_{u^*})(s) \left( u^i_k(s) - (u^*)^i(s) \right) \; \mathrm{d}s \right\|}_{\triangleq \delta^{u,2}_{k+1}(t)} ,
$$}
where the (overloaded) constant $C_4 \ge 0$ comes from the uniform boundedness of both $(x_k)_{k \in \mathbb{N}}$ and $(\tilde z^{r_k,v}_k)_{k \in \mathbb{N}}$ previously stated. In particular, we introduce the terms \revThird{$F(f_i,L^i,\tilde z^{r,v}_{u^*}) : [r,t_f] \rightarrow \mathbb{R}$} that are continuous and uniformly bounded mappings depending on $f_i$, $L^i$, and $\tilde z^{r,v}_{u^*}$. Following the exact same argument we developed for $\delta^{u,1}_{k+1}$, one prove that $\delta^{u,2}_{k+1}(t) \rightarrow 0$ for $k \rightarrow \infty$, uniformly in the interval \revThird{$[r,t_f]$}, so that \eqref{eq_convTraj} and a routine Gr\"onwall inequality argument allow us to obtain \eqref{eq_convVarTraj}.

Importantly, convergence \eqref{eq_convVarTraj} implies that, for $k \rightarrow \infty$, 
\begin{equation} \label{eq_convVar}
    \revThird{\| \tilde z^{r_k,v}_k(t_f) - \revThird{( \tilde z^{r,v}_{u^*}(t_f) , 0 )} \| \longrightarrow 0 .}
\end{equation}

\subsubsection{Convergence of Extremals and Conclusion}

At this step, consider the sequence of tuples \revThird{$(\mathfrak{p}_k,p^0_k,p^1_k)_{k \in \mathbb{N}}$}, with $p^0_k \le 0$ for $k \ge 1$. It is clear that the variational expressions \eqref{ref:variationalLOCP} remain valid whenever \revThird{$(\mathfrak{p}_k,p^0_k,p^1_k)$} is multiplied by some positive constant. Therefore, without loss of generality, we may assume that \revThird{$\| (\mathfrak{p}_k,p^0_k,p^1_k) \| = 1$} and $p^0_k \le 0$ for $k \ge 1$. Then, we can extract a subsequence (still denoted $(\mathfrak{p}_k,p^0_k,p^1_k)_{k \in \mathbb{N}}$) that converges to some nontrivial tuple $(\mathfrak{p},p^0,p^1)$ satisfying $p^0 \le 0$. At this step, we may leverage \eqref{eq_convTraj} and \eqref{eq_convVar} to prove that $(t^*_f,x^*,\mathfrak{p},p^0,u^*)$ is the sought-after \revThird{non-trivial} extremal for OCP \revThird{when $p^0 \neq 0$}. Indeed, for every $r \in (0,t^*_f)$ Lebesgue point of $u^*$, and $v \in U$, \eqref{eq_convTraj} and \eqref{eq_convVar} we have that, for $k \rightarrow \infty$,
\begin{align*}
    &\revThird{\left( p^0 , \mathfrak{p} \frac{\partial g}{\partial x}(x^*(t_f)) \right) \cdot \tilde z^{r,v}_{u^*}(t_f) = } \\
    &\revThird{\left( p^0 , \mathfrak{p} \frac{\partial g}{\partial x}(x^*(t_f)) , p^1 \right) \cdot ( \tilde z^{r,v}_{u^*}(t_f) , 0 ) \le} \\
    &\revThird{\le \bigg| \left( p^0 , \mathfrak{p} \frac{\partial g}{\partial x}(x^*(t_f)) , p^1 \right) \cdot ( \tilde z^{r,v}_{u^*}(t_f) , 0 )} \\
    &\revThird{\qquad \qquad \qquad - \left( p^0_k , \mathfrak{p}_k \frac{\partial g}{\partial x}(x_k(t_f)) , p^1_k \right) \cdot \tilde z^{r_k,v}_k(t_f) \bigg| \longrightarrow 0}
\end{align*}
due to the inequality of \eqref{ref:variationalLOCP}, and we conclude.

The proof of Theorem \ref{theo_main} is achieved if we show that
\begin{equation} \label{eq_convAdj}
    \revThird{\underset{s \in [0,t_f]}{\sup} \ \| p_{k+1}(s) - p(s) \| \longrightarrow 0}
\end{equation}
for $k \rightarrow \infty$, where $p_{k+1}$ solves
{\small
\begin{equation*}
    \begin{cases}
        \displaystyle \dot{p}_{k+1}(s) = -\frac{\partial H_{k+1}}{\partial x}(s,x_{k+1}(s),p_{k+1}(s),p^0_{k+1},u_{k+1}(s)) \medskip \\
        \revThird{\displaystyle p_{k+1}(t_f) = \mathfrak{p}_{k+1} \frac{\partial g}{\partial x}(x_{k+1}(t_f)) ,}
    \end{cases}
\end{equation*}}
whereas $p$ solves
\begin{equation*}
    \begin{cases}
        \displaystyle \dot{p}(s) = -\frac{\partial H}{\partial x}(s,x^*(s),p(s),p^0,u^*(s)) \medskip \\
        \revThird{\displaystyle p(t_f) = \mathfrak{p} \frac{\partial g}{\partial x}(x^*(t_f)) .}
    \end{cases}
\end{equation*}
To this end, by leveraging the uniform boundedness of the trajectories, with the same exact argument proposed in the proof of Lemma \ref{lemma_boundness} (see the Appendix), one shows that the sequence $(p_k)_{k \in \mathbb{N}}$ is uniformly bounded in the interval \revThird{$[0,t_f]$}. From this, for every \revThird{$t \in [0,t_f]$} we have that
{\small
\begin{align*}
    &\revThird{\| p_{k+1}(t) - p(t) \| \le \left\| \mathfrak{p}_{k+1} \frac{\partial g}{\partial x}(x_{k+1}(t_f)) - \mathfrak{p} \frac{\partial g}{\partial x}(x^*(t_f)) \right\|} \\
    &\revThird{+ C_5 \bigg( | p^1_{k+1} | \Delta^{\frac{1}{2}}_{k+1} + | p^0_{k+1} - p^0 | + \int^{t_f}_t \| p_{k+1}(s) - p(s) \| \; \mathrm{d}s} \\
    &\revThird{\qquad + \int^{t_f}_t \| x_k(s) - x^*(s) \| \; \mathrm{d}s + \int^{t_f}_t \| x_{k+1}(s) - x^*(s) \| \; \mathrm{d}s \bigg)} \\
    &\revThird{+ \underbrace{\sum^m_{i=1} \left\| \int^{t_f}_t F(f_i,L^i,p,p^0)(s) \left( u^i_k(s) - (u^*)^i(s) \right) \; \mathrm{d}s \right\|}_{\triangleq \delta^{u,3}_{k+1}(t)}}
\end{align*}}
where the constant $C_5 \ge 0$ comes from the uniform boundedness of both $(x_k)_{k \in \mathbb{N}}$ and $(p_k)_{k \in \mathbb{N}}$, whereas \revThird{$F(f_i,L^i,p,p^0) : [r,t_f] \rightarrow \mathbb{R}$} again denote continuous and uniformly bounded mappings that depend on $f_i$, $L^i$, $p$, and $p^0$. Following the exact same argument we developed for $\delta^{u,1}_{k+1}$, one proves that $\delta^{u,3}_{k+1}(t) \rightarrow 0$ for $k \rightarrow \infty$, uniformly in the interval \revThird{$[0,t_f]$}, so that \eqref{eq_convTraj} and a routine Gr\"onwall inequality argument allow us to conclude (see also the Appendix).

\section{Sequential Convex Programming with \\ Manifold-Type Constraints} \label{sec_manif}

We now show how the framework described in Section \ref{sec_noManif} can be applied \textit{verbatim} to solve our optimal control problem when additional manifold-type constrains are considered, under mild regularity assumptions on the dynamics. In this context, we focus on problems OCP$^M$ defined as:
\begin{equation*}
    \begin{cases}
        \displaystyle \underset{\revThird{u \in \mathcal{U}}}{\min} \ \int^{t_f}_0 f^0(s,x(s),u(s)) \; \mathrm{d}s \\
        \dot{x}(s) = f(s,x(s),u(s)) , \ \textnormal{a.e.} \ \revThird{s \in [0,t_f]} \\
        x(0) = x^0 \in M , \quad g(x(t_f)) = 0 \\
        x(s) \in M \subseteq \mathbb{R}^n , \quad s \in [0,t_f]
    \end{cases}
\end{equation*}
where $M \subseteq \mathbb{R}^n$ is a smooth $d$-dimensional submanifold of $\mathbb{R}^n$ and, for the sake of consistency, we assume that $g^{-1}(0) \cap M \neq \emptyset$. \revThird{We denote by $(x^*,u^*)$ any solution to OCP$^M$.} 

\subsection{Unchanged Framework under Regular Dynamics}

One possibility to solve OCP$^M$ would consist of penalizing the manifold-type constraints within the cost (see Remark \ref{remark_penalization}). Although possible, this approach might add undue complexity to the formulation. Interestingly, in several important cases for applications, this issue can be efficiently avoided. To this end, we assume that the following \textit{regularity condition} holds:
\begin{Itemize}
    \item[$(A_3)$] For $i=0,\dots,m$, the vector fields $f_i : \mathbb{R}^{n+1} \rightarrow \mathbb{R}^n$ are such that $f_i(s,x) \in T_x M$, for every $(s,x) \in \mathbb{R} \times M$.
\end{Itemize}
In $(A_3)$, $T_x M$ denotes the tangent space of $M$ at $x \in M$, which we identify with a $d$-dimensional subspace of $\mathbb{R}^n$. This requirement is often satisfied when dealing with mechanical systems in aerospace and robotics applications (for instance, consider rotation and/or quaternion-type constraints). Under $(A_3)$, as a classical result, the trajectories of $\dot{x}(s) = f(s,x(s),u(s))$ starting from $x^0 \in M$ lie on the submanifold $M$, and therefore, the condition $x(s) \in M$, $s \in [0,t_f]$, is automatically satisfied. In other words, we may remove manifold-type constraints from problem OCP$^M$ so that it exactly resembles OCP, i.e., the formulation adopted in Section \ref{sec_noManif} with the additional constraint $x^0 \in M$. At this step, we may leverage the machinery built previously to solve OCP. Specifically, the construction of each subproblem \revThird{LOCP$^{\Delta}_k$} and Algorithm \ref{algo_SCP} applies unchanged. Due to the linearization of the dynamics, solutions to the convex subproblems are not supposed to lie on $M$. However, convergence does force the limiting trajectory to satisfy the manifold-type constraints.

\subsection{Convergence Analysis}

The convergence of Algorithm \ref{algo_SCP} applied to this new context can be inferred from Theorem \ref{theo_main}. However, despite the regularity assumption $(A_3)$, it is not obvious that the optimality claimed by this result extends to the general geometric setting brought on by manifold-type constraints. Specifically, if Algorithm \ref{algo_SCP} converges to a trajectory satisfying the assumptions of Theorem \ref{theo_main}, although such a trajectory meets manifold-type constraints, the related extremal satisfies the PMP for problems defined in the Euclidean space by construction. In other words, a priori the extremal does not carry any information about the geometric structure of a problem with manifold-type constraints, whereas extremals for OCP$^M$ are expected to satisfy stronger \textit{geometrically-consistent} necessary conditions for optimality. Specifically, to recover a geometrically-consistent candidate optimal solution for OCP$^M$, we must show that this satisfies the Geometric PMP (GPMP) (see, e.g., \cite{AgrachevSachkov2004}), necessary conditions for optimality for OCP$^M$ which are stronger than PMP. This is our next objective.

Before stating the GPMP related to formulation OCP$^M$, we first need to introduce some notation and preliminary results (further details may be found in \cite{AgrachevSachkov2004}). We denote $TM$ and $T^*M$ as the tangent and cotangent bundle of $M$, respectively. Due to $(A_3)$, the mapping
$$
f_M : \mathbb{R} \times M \times \mathbb{R}^m \rightarrow TM : (s,x,u) \mapsto f(s,x,u) $$
is a well-defined, non-autonomous vector field of $M$. Thus, trajectories related to feasible solutions $(t_f,x,u)$ for OCP$^M$ may be seen as solutions to the geometric dynamical equations
\begin{equation} \label{eq_geoDyn}
    \dot{x}(s) = f_M(s,x(s),u(s)) , \quad x^0 \in M .
\end{equation}
In a geometric setting, given a feasible solution \revThird{$(x,u)$} for OCP$^M$, Pontryagin extremals are represented by the quantity \revThird{$(\lambda,p^0,u)$}. In particular, the information concerning the trajectory $x$ that satisfies \eqref{eq_geoDyn} is encapsulated within the cotangent curve $\lambda : [0,t_f] \rightarrow T^*M$, i.e., $x(s) = \pi(\lambda(s))$, $s \in [0,t_f]$, where $\pi : T^*M \rightarrow M$ is the canonical projection. At this step, for $\lambda \in T^*M$ and $p^0 \in \mathbb{R}$, we may define the geometric Hamiltonian (related to OCP$^M$) as
$$
H(s,\lambda,p^0,u) \triangleq \langle \lambda , f_M(s,\pi(\lambda),u) \rangle + p^0 f^0(s,\pi(\lambda),u),
$$
where $\langle \cdot, \cdot \rangle$ denotes the duality in $T^*M$. We remark that whenever $M = \mathbb{R}^n$, we recover the Hamiltonian introduced in Section \ref{sec_noManif}. In the geometric framework, adjoint equations are described in terms of \textit{Hamiltonian vector fields}. Specifically, as a classical result, for every $(s,u) \in \mathbb{R}^{m+1}$ one can associate to $H(s,\cdot,\cdot,u)$ a unique vector field $\overset{\rightarrow}{H}(s,\cdot,\cdot,u) : T^*(M \times \mathbb{R}) \rightarrow T(T^*(M \times \mathbb{R}))$ of the product cotangent bundle $T^*(M \times \mathbb{R})$ (known as Hamiltonian vector field) by the rule $\displaystyle \sigma_{(\lambda,p^0)}\Big(\cdot,\overset{\rightarrow}{H}(s,\lambda,p^0,u)\Big) = \frac{\partial H}{\partial (\lambda,p^0)}(s,\lambda,p^0,u)$, with $\sigma$ being the canonical symplectic form of $T^*(M \times \mathbb{R})$. We are now ready to state the GPMP related to OCP$^M$.

\begin{thm}[\revThird{GPMP for OCP$^M$ \cite{AgrachevSachkov2004}}] \label{theo_PMPGeo}
Let \revThird{$(x^*,u^*)$} be a locally-optimal solution to OCP$^M$. There exists an absolutely continuous curve\footnote{Continuity is meant with respect to the standard topology in $T^*M$.} \revThird{$\lambda : [0,t_f] \rightarrow T^*M$} with $x^*(s) = \pi(\lambda(s))$, \revThird{$s \in [0,t_f]$} and a constant $p^0 \le 0$, such that the following hold:
\begin{itemize}
    \item Non-Triviality Condition: $(\lambda,p^0) \neq 0$.
    \item Adjoint Equation: Almost everywhere in \revThird{$[0,t_f]$},
    $$
    \frac{d(\lambda,p^0)}{ds}(s) = \overset{\rightarrow}{H}(s,\lambda(s),p^0,u^*).
    $$
    \item Maximality Condition: Almost everywhere in \revThird{$[0,t_f]$},
    $$
    H(s,\lambda(s),p^0,u^*(s)) = \underset{v \in U}{\max} \ H(s,\lambda(s),p^0,v).
    $$
    \item Transversality Condition: It holds that
    $$
    \revThird{\lambda(t_f) \ \perp \ \textnormal{ker} \frac{\partial g_M}{\partial x}(x^*(t_f)) ,}
    $$
    where we denote $g_M : M \rightarrow \mathbb{R}^{\ell_g} : x \mapsto g(x)$.
\end{itemize}
The tuple \revThird{$(\lambda,p^0,u^*)$} is called geometric extremal \revThird{for OCP$^M$}.
\end{thm}

\revThird{\begin{rmk}
    As discussed previously, thanks to $(A_3)$, each convex problem LOCP$^{\Delta}_k$ may be correctly formulated as in Section \ref{sec_convexProb} by dropping manifold-type constraints and considering dynamics as vector fields in $\mathbb{R}^n$. Therefore, we may again leverage Theorem \ref{theo_PMPk} as necessary conditions for optimality for each LOCP$^{\Delta}_k$. Accordingly, for $k \ge 1$, (weak) extremals for LOCP$^{\Delta}_k$ are denoted by $(x_k,p_k,p^0_k,p^1_k,u_k)$.
\end{rmk}}

Assuming that Algorithm \ref{algo_SCP} applied as described above converges, we prove that the limiting solution is a candidate local optimum for OCP$^M$ by showing that it is possible to appropriately \textit{orthogonally project} the extremal for OCP provided by Theorem \ref{theo_main} to recover a geometric extremal for OCP$^M$. First, we need to introduce the notion of the orthogonal projection to a subbundle. Specifically, given the cotangent bundles $T^*M \subseteq T^*\mathbb{R}^n \cong \mathbb{R}^{2 n}$, define $T^*\mathbb{R}^n|_M \triangleq \bigcup_{x \in M} \ \{ x \} \times T^*_x \mathbb{R}^n \cong M \times \mathbb{R}^n$. Equipped with the structure of the pullback bundle given by the canonical projection $T^*\mathbb{R}^n|_M \rightarrow M$, $T^*\mathbb{R}^n|_M$ is a vector bundle over $M$ of rank $n$, and $T^*M$ may be identified with a subbundle of $T^*\mathbb{R}^n|_M$. We build an orthogonal projection operator from $T^*\mathbb{R}^{n+1}|_{\mathbb{R} \times M}$ to $T^*(\mathbb{R} \times M)$ by leveraging the usual orthogonal projection in $\mathbb{R}^{n+1}$. To do this, let $x \in M$ and $(V,\varphi) = (V,y^1,\dots,y^n)$ be a local chart of $x$ in $\mathbb{R}^n$ adapted to $M$, i.e., satisfying $\varphi(V \cap M) = \varphi(V) \cap \mathbb{R}^{d} \times \{ 0 \}^{n-d}$. By construction, $\{ dy^j(\cdot) \}_{j=1,\dots,n}$ is a local basis for $T^*\mathbb{R}^n|_M$ and $\{ dy^j(\cdot) \}_{j=1,\dots,d}$ is a local basis for $T^*M$ around $x$. Consider the cometric $\langle \cdot , \cdot \rangle_{\mathbb{R}^n}$ in $T^*\mathbb{R}^n|_M$ which is induced by the Euclidean scalar product in $\mathbb{R}^n$. The Gram-Schmidt process applied to $\{ dy^j(\cdot) \}_{j=1,\dots,n}$ provides a local orthonormal frame $\{ E_j(\cdot) \}_{j=1,\dots,n}$ for $T^*\mathbb{R}^n|_M$, that satisfies in $V \cap M$
\begingroup
\begin{equation} \label{ref:Span}
\textnormal{span} \langle E_1(\cdot) , \dots , E_j(\cdot) \rangle = \textnormal{span} \langle dy^1(\cdot) , \dots , dy^j(\cdot) \rangle
\end{equation}
\endgroup
for every $1 \le j \le n$. It follows that, when restricted to $V \cap M$, the following orthogonal projection operator
\begingroup
\begin{equation*}
\begin{split}
\Pi : \ &T^* \mathbb{R}^{n+1}|_{\mathbb{R} \times M} \rightarrow T^*(\mathbb{R} \times M) \cong \mathbb{R}^2 \times T^*M \\
&(z,x,p^0,p) \mapsto \bigg( (z,p^0) , \sum_{j=1}^d \langle p , E_j(x) \rangle_{\mathbb{R}^n} \ E_j(x) \bigg)
\end{split}
\end{equation*}
\endgroup
is well-defined and smooth. Moreover, since the change of frame mapping between two orthonormal frames is orthogonal, from \eqref{ref:Span} it is readily checked that $\Pi$ is globally defined. Equipped with the GPMP and orthogonal projections, the numerical strategy to solve OCP$^M$ detailed above becomes meaningful and justified by the following convergence result (similar to the discussion for Theorem \ref{theo_main}, the convergences stated therein readily extend to the discretized setting).

\begin{thm}[{\small Convergence for SCP with manifold constraints}] \label{theo_mainGeo}
Assume that $(A_1)$, $(A_2)$, and $(A_3)$ hold\revThird{, and that applying Algorithm \ref{algo_SCP} to OCP$^M$ when manifold-type constraints are dropped returns a sequence $(\Delta_k,u_k,x_k)_{k \in \mathbb{N}}$ such that $\Delta_k \rightarrow 0$ and, for every $k \in \mathbb{N}$, the tuple $(u_{k+1},x_{k+1})$} locally solves LOCP$^{\Delta}_{k+1}$. Then there exists a tuple \revThird{$(x^*,p,p^0,u^*)$} that is an extremal for OCP$^M$ when manifold-type constraints are dropped and satisfies all the statements listed in Theorem \ref{theo_main}\revThird{, if $(p(t_f),p^0) \neq 0$} (where the convergence of $(u_k)_{k \in \mathbb{N}}$ for the strong topology of $L^2$ may be replaced by the weak topology of $L^2$ whenever $(A_2)$ holds \revThird{and, for $k \ge 1$, each multiplier $(x_k,p_k,p^0_k,p^1_k,u_k)$ for LOCP$^{\Delta}_k$ satisfies $p^0_k \neq 0$}). In addition, the limiting trajectory satisfies $x^*(s) \in M$, \revThird{$s \in [0,t_f]$}, and by defining the absolutely continuous curve
\begin{equation} \label{eq_lambda}
\hspace{-2.5ex} \revThird{\lambda : [0,t_f] \rightarrow T^*M : t \mapsto \pi_2\Big( \Pi\big(z^*(t),x^*(t),p^0,p(t)\big) \Big) ,}
\end{equation}
where $\pi_2 : T^*(\mathbb{R} \times M) \rightarrow T^*M : ((z,p^0),\xi) \mapsto \xi$ and \revThird{$z^* : [0,t_f] \rightarrow \mathbb{R}$} satisfies $\dot{z}(s) = f^0(s,x^*(s),u^*(s))$, $z(0) = 0$, the tuple \revThird{$(\lambda,p^0,u^*)$} is a geometric extremal for OCP$^M$.
\end{thm}

\subsection{Proof of the Convergence Result}

Let \revThird{$(x^*,p,p^0,u^*)$} be an extremal for OCP$^M$ in the case where manifold-type constraints are dropped, whose existence is guaranteed by Theorem \ref{theo_main} \revThird{(in particular, we assume that $(p(t_f),p^0) \neq 0$)}. To avoid overcharging the notation, in the rest of this section we denote \revThird{$(x,p,p^0,u) \triangleq (x^*,p,p^0,u^*)$}. Because $(A_3)$ implies that $x(s) \in M$, $s \in [0,t_f]$, Theorem \ref{theo_mainGeo} is proved once we show that the tuple \revThird{$(\lambda,p^0,u)$} with $\lambda$ built as in \eqref{eq_lambda} satisfies the non-triviality condition, the adjoint equation, the maximality condition, and the transversality condition of Theorem \ref{theo_PMPGeo}. For the sake of clarity, we denote $\displaystyle dg_x = \frac{\partial g}{\partial x}(x)$, $\displaystyle d(g_M)_x = \frac{\partial g_M}{\partial x}(x)$.

\subsubsection{Adjoint Equation}

Before getting started, we introduce some fundamental notations. For every $(t_0,z_0,p_0) \in [0,t_f] \times \mathbb{R}^{n+1}$, the differential equation
\begin{equation} \label{eq_dynSyst}
    \dot{z}(s) = f^0(s,x(s),u(s)) , \ \dot{x}(s) = f(s,x(s),u(s))
\end{equation}
with $z(t_0) = z_0$, $x(t_0) = p_0$ has a unique solution, which may be extended to the whole interval $[0,t_f]$. We denote by $\exp : [0,t_f]^2 \times \mathbb{R}^{n+1} \rightarrow \mathbb{R}^{n+1}$ the flow of \eqref{eq_dynSyst}, i.e., $\exp(\cdot;t_0,(z_0,p_0))$ solves \eqref{eq_dynSyst} with initial condition $(z_0,p_0)$ at time $t_0$. As a classical result, for every $(t,t_0) \in [0,t_f]^2$, the mapping $\exp(t;t_0,\cdot) : \mathbb{R}^{n+1} \rightarrow \mathbb{R}^{n+1}$ is a diffeomorphism. With this notation at hand, one can show that the solution $p$ to the adjoint equation of Theorem \ref{theo_main} is such that for $s \in [0,t_f]$,
\begin{equation} \label{eq_expAdj}
    (p^0,p(s)) = (\exp(t_f;s,\cdot))^*_{(z,x)(t_f)} \cdot (p^0,p(t_f)),
\end{equation}
where we denote $(z,x)(t) \triangleq \exp(t,0;(0,x^0))$ and $(\cdot)^*$ denotes the pullback operator of 1-forms in $\mathbb{R}^{n+1}$ (see, e.g., \cite{AgrachevSachkov2004}). At this step, to prove that $(\lambda,p^0)$ satisfies the adjoint equation of Theorem \ref{theo_mainGeo} with $\lambda$ defined in \eqref{eq_lambda} and $(p^0,p)$ satisfying \eqref{eq_expAdj}, we can leverage classical results from symplectic geometry in the context of Hamiltonian equations (see, e.g., \cite{AgrachevSachkov2004}) from which it is sufficient to prove the following lemma:

\begin{lmm}[Projections of solutions to Hamiltonian systems] \label{lemma_hamilton}
For almost every $t \in [0,t_f]$, let $(V,\varphi) = (V,y^0,\dots,y^n)$ be a local chart of $(z,x)(t) \triangleq \exp(t,0;(0,x^0))$ (which is a point in $\mathbb{R} \times M$ due to $(A_4)$) in $\mathbb{R}^{n+1}$ adapted to $\mathbb{R} \times M$. For every $i=0,\dots,d$, it holds that
{\small
\begin{align*}
&\frac{d}{ds} \bigg( \Pi \Big( (\exp(t_f;s,\cdot))^*_{(z,x)(t_f)} \cdot (p^0,p(t_f)) \Big) \bigg( \frac{\partial}{\partial y^i}\big( (x,z)(s) \big) \bigg) \bigg)(t) \\
&= -\sum_{j=0}^d \frac{\partial (f^0,f_M)_j}{\partial y^i}(t,x(t),u(t)) \ \cdot \\
&\qquad \qquad \Pi \Big( (\exp(t_f;t,\cdot))^*_{(z,x)(t_f)} \cdot (p^0,p(t_f)) \Big)
\bigg( \frac{\partial}{\partial y^j}\big( (x,z)(t) \big) \bigg) ,
\end{align*}}
where $(\cdot)^*$ denotes the pullback operator of 1-forms in $\mathbb{R}^{n+1}$.
\end{lmm}

\begin{proof}
For indices $i=0,\dots,n$, we denote
{\small
$$
a_i(t) = \Pi \Big( (\exp(t_f;t,\cdot))^*_{(z,x)(t_f)} \cdot (p^0,p(t_f)) \Big)
\bigg( \frac{\partial}{\partial y^j}\big( (x,z)(t) \big) \bigg) .
$$}
Since by the definition of the pullback it holds that
{\small
\begin{equation} \label{ref:linearCombinationAdjoint}
    (\exp(t_f;t,\cdot))^*_{(z,x)(t_f)} \cdot (p^0,p(t_f)) = \sum_{j=0}^n b_j(t) dy^j\big((z,x)(t)\big)
\end{equation}}

\noindent for appropriate coefficients $b_j(t)$, $j = 0,\dots,n$, from \eqref{ref:Span},

{\small
\begin{equation*}
    \Pi \big( (\exp(t_f;t,\cdot))^*_{(z,x)(t_f)} \cdot (p^0,p(t_f)) \big) = \sum_{j=0}^d b_j(t) dy^j\big((z,x)(t)\big) ,
\end{equation*}}
which yields $a_j(t) = b_j(t)$ for every $j = 0,\dots,d$. Therefore, by inverting \eqref{ref:linearCombinationAdjoint}, we obtain
\begin{align*}
(p^0,p(t_f)) &= \sum_{j=0}^d a_j(t) (\exp(t;t_f,\cdot))^*_{(z,x)(t)} \cdot dy^j\big( (z,x)(t) \big) \\
&+ \sum_{j=d+1}^n b_j(t) (\exp(t;t_f,\cdot))^*_{(z,x)(t)} \cdot dy^j\big( (z,x)(t) \big) .
\end{align*}
Now, let $(A,\alpha) = (A,w^0,\dots,w^n)$ be a local chart of $(z,x)(t_f)$ in $\mathbb{R}^{n+1}$ adapted to $\mathbb{R} \times M$. Since due to $(A_4)$, the trajectory $(z,x)(t)$ lies entirely in $\mathbb{R} \times M$ and the chart $(V,\varphi)$ is adapted to $\mathbb{R} \times M$, for every $i=0,\dots,d$ and every $j \ge d+1$, one computes
\begin{equation*}
\begin{split}
&(\exp(t;t_f,\cdot))^*_{(z,x)(t)} \cdot dy^j\big( (z,x)(t) \big) \bigg( \frac{\partial}{\partial w^i}\big( (z,x)(t_f) \big) \bigg) \\
&= \frac{\partial}{\partial w^i} (y^j \circ \exp(t;t_f,\cdot) \circ \alpha^{-1})\big( \alpha((z,x)(t_f)) \big) = 0 .
\end{split}
\end{equation*}
This implies that for every $i=0,\dots,d$,
\begin{multline*}
(p^0,p(t_f)) \bigg( \frac{\partial}{\partial w^i}\big( (z,x)(t_f) \big) \bigg) = \\
\sum_{j=0}^d a_j(t) \frac{\partial}{\partial w^i} (y^j \circ \exp(t;t_f,\cdot) \circ \alpha^{-1})\big( \alpha((z,x)(t_f)) \big) .
\end{multline*}
The term on the left-hand side does not depend on $t$. Therefore, a differentiation with respect to $t$ together with \eqref{eq_expAdj} lead to\footnote{Note that, as long as $i = 0,\dots,d$, quantities in \eqref{ref:eqFinalLemma} evolve in $\mathbb{R} \times M$. Therefore, indices greater than $d$ do not explicitly appear in calculations.}
\begingroup
\begin{multline} \label{ref:eqFinalLemma}
\sum_{j=0}^d \bigg[ \dot{a}_j(t) \bigg( (\exp(t;t_f,\cdot))^*_{(z,x)(t)} \cdot \\
dy^j\big( (z,x)(t) \big) \bigg( \frac{\partial}{\partial w^i}\big( (z,x)(t_f) \big) \bigg) \bigg) \\
+ \sum_{\ell=0}^d a_j(t) \frac{\partial (f^0,f)_j}{\partial y^{\ell}}(t,x(t),u(t)) \bigg( (\exp(t;t_f,\cdot))^*_{(z,x)(t)} \cdot \\
dy^{\ell}\big( (z,x)(t) \big) \bigg( \frac{\partial}{\partial w^i}\big( (z,x)(t_f) \big) \bigg) \bigg) \bigg] = 0,
\end{multline}
\endgroup
which must hold for every $i=0,\dots,n$. At this step, we notice that due to $(A_4)$, for every $j=0,\dots,d$ and every $\ell=0,\dots,d$, we have $\displaystyle \frac{\partial (f^0,f)_j}{\partial y^{\ell}}(t,x(t),u(t)) = \displaystyle \frac{\partial (f^0,f_M)_j}{\partial y^{\ell}}(t,x(t),u(t))$. Moreover, due to $(A_4)$, the restriction $\exp(t;t_f,\cdot) : \mathbb{R} \times M \rightarrow \mathbb{R} \times M$ is well-defined and is a diffeomorphism. Hence $(\exp(t;t_f,\cdot))^*$ is an isomorphism when restricted to 1-forms in $T^*M$. Combining those with \eqref{ref:eqFinalLemma} gives
\begingroup
\begin{multline*}
\sum_{j=0}^d \bigg[ \dot{a}_j(t) dy^j\big( (z,x)(t) \big) \\
+ \sum_{\ell=0}^d a_j(t) \frac{\partial (f^0,f_M)_j}{\partial y^{\ell}}(t,x(t),u(t)) dy^{\ell}\big( (z,x)(t) \big) \bigg] = 0
\end{multline*}
\endgroup
and the conclusion follows.
\end{proof}

\subsubsection{Maximality, Transversality and Non-Triviality Conditions}

Before getting started, consider the following analysis of tangent spaces. From $(A_3)$ and the definition of $g_M$, it holds that $x(t_f) \in g^{-1}(0) \cap M = g^{-1}_M(0)$. Note that $g^{-1}(0) \subseteq \mathbb{R}^n$ and $g^{-1}_M(0) \subseteq M$ are submanifolds of dimension $n-\ell_g$ and $d-\ell_g$, respectively, with tangent spaces given by
$$
T_x g^{-1}(0) = \{ v \in T_x \mathbb{R}^n \cong \mathbb{R}^n : dg_x(v) = 0 \} , \ x \in g^{-1}(0)
$$
$$
T_x g^{-1}_M(0) = \{ v \in T_x M : d(g_M)_x(v) = 0 \} , \quad x \in g^{-1}_M(0) .
$$
In particular, by subspace identification, for every $x \in g^{-1}_M(0)$, one has $T_x g^{-1}_M(0) \subseteq T_x g^{-1}(0) \cap T_x M$. The inclusion above is actually an identity. To see this, let $x \in g^{-1}_M(0) \subseteq M$ and $(V,\varphi) = (V,y^1,\dots,y^n)$ be a local chart of $x$ in $\mathbb{R}^n$ adapted to $M$. The definition of adapted charts immediately gives that $\displaystyle dg_x\left( \frac{\partial}{\partial y^j}(x) \right) = d(g_M)_x\left( \frac{\partial}{\partial y^j}(x) \right)$, for $j = 1,\dots,d$. Thus, if $\displaystyle v = \sum^d_{j = 1} v^j \frac{\partial}{\partial y^j}(x) \in T_x M$ such that $dg_x(v) = 0$, we have
$$
d(g_M)_x(v) = \sum^d_{j = 1} v^j d(g_M)_x\left( \frac{\partial}{\partial y^j}(x) \right) = dg_x(v) = 0,
$$
and the sought after identity follows. A straightforward application of Grassmann's formula to this identity yields
\begin{equation} \label{eq_Grassmann}
    \mathbb{R}^n = T_{x(t_f)} g^{-1}(0) + T_{x(t_f)} M .
\end{equation}

\revThird{Noticing that the maximality condition is a straightforward consequence of $(A_3)$, we are now ready to prove the transversality condition and the non-triviality condition.}

To show the validity of the transversality condition, let us prove that for every $v \in T_{x(t_f)} M \subseteq \mathbb{R}^n$, it holds that
\begin{equation} \label{eq_perpProof}
    \langle \lambda(t_f) , v \rangle = p(t_f)^{\top} v ,
\end{equation}
which provides the desired result because $p(t_f) \cdot v = 0$ for $v \in T_{x(t_f)} g^{-1}(0)$, due to Theorem \ref{theo_main}. To show this, by the Gram-Schmidt process, we can build a local orthonormal frame $\{ E_j(\cdot) \}_{j=1,\dots,n}$ for $T^*_{x(t_f)} \mathbb{R}^n$ around $x(t_f)$ such that $\{ E_j(\cdot) \}_{j=1,\dots,d}$ is a local frame for $T^*_{x(t_f)} M$ around $x(t_f)$. The dual frames $\{ E^*_j(x(t_f)) \}_{j=1,\dots,n}$ and $\{ E^*_j(x(t_f)) \}_{j=1,\dots,d}$ span $T^{**}_{x(t_f)} \mathbb{R}^n \cong T_{x(t_f)} \mathbb{R}^n \cong \mathbb{R}^n$ and $T^{**}_{x(t_f)} M \cong T_{x(t_f)} M$, respectively. Thus, for any tangent vector $v \in T_{x(t_f)} M$, the definitions of the dual frame and of the orthogonal projection $\Pi$ allow us to conclude that
\begin{align*}
    &\langle \lambda(t_f) , v \rangle = \left\langle \sum_{j=1}^d \langle p , E_j(x) \rangle_{\mathbb{R}^n} \ E_j(x) , \sum_{j=1}^d v^j \ E^*_j(x) \right\rangle \\
    & = \left\langle \sum_{j=1}^n \langle p , E_j(x) \rangle_{\mathbb{R}^n} \ E_j(x) , \sum_{j=1}^d v^j \ E^*_j(x) \right\rangle = p(t_f)^{\top} v .
\end{align*}

Finally, let us focus on the non-triviality condition. By contradiction, assume that there exists $t \in [0,t_f]$ such that $(\lambda(t),p^0) = 0$. The linearity of the adjoint equation yields $\lambda(s) = 0$ for all $s \in [0,t_f]$, so that $\lambda(t_f) = 0$. On the other hand, from the transversality condition of Theorem \ref{theo_main}, we know that $p(t_f) \perp T_{x(t_f)} g^{-1}(0)$. Now, given $v \in \mathbb{R}^n$, from \eqref{eq_Grassmann} we infer that $v = v_1 + v_2$ with $v_1 \in T_{x(t_f)} g^{-1}(0)$ and $v_2 \in T_{x(t_f)} M$ so that from \eqref{eq_perpProof}, one obtains
$$
p(t_f)^{\top} v = p(t_f)^{\top} v_2 = \langle \lambda(t_f) , v_2 \rangle = 0 .
$$
This leads to $(p,p^0) = 0$, in contradiction with the non-triviality condition of Theorem \ref{theo_main}. The conclusion follows.

\section{Extension to Problems with Free Final Time}

\revThird{Algorithm \ref{algo_SCP} and related convergence guarantees, i.e., Theorems \ref{theo_main} and \ref{theo_mainGeo}, can be suitably extended to the case of OCP (and OCP$^M$) with free final time $t_f \in [0,T]$, where $T > 0$ is a given upper bound, as we quickly summarize in this section.}

\revThird{By adapting the notation introduced in Section \ref{sec_noManif} accordingly, each convex problem LOCP$^{\Delta}_k$ can be formulated as presented in Section \ref{sec_convexProb}, except for adding the additional trust-region constraint $| t^{k+1}_f - t^k_f | \le \Delta_{k+1}$ to \eqref{eq_trust}. For what concerns the theoretical analysis of the convergence of SCP in such setting, our original assumptions require some major modification. Specifically, the presence of free final times adds an additional pointwise transversality condition in Theorems \ref{theo_PMP}, \ref{theo_PMPk}, and \ref{theo_PMPGeo}, which in turn requires to replace $(A_2)$ with the following stronger requirement:}
\revThird{\begin{Itemize}
    \item[$(\tilde A_2)$] For every $k \in \mathbb{N}$, the following conditions hold:
    \begin{itemize}
        \item any optimal solution $(t^{k+1}_f,x_{k+1},u_{k+1})$ to LOCP$^{\Delta}_{k+1}$ satisfies:
        $$
        | t^{k+1}_f - t^k_f | + \int^T_0 \| x_{k+1}(s) - x_k(s) \|^2 \; \mathrm{d}s < \Delta_{k+1} ;
        $$
        \item any optimal control $u_{k+1}$ to LOCP$^{\Delta}_{k+1}$ is continuous.
    \end{itemize}
\end{Itemize}}
\noindent \revThird{In \cite{BonalliArxiv2021}, through an analysis which is similar to the one proposed in the present paper, under $(\tilde A_2)$ we showed that the statements of Theorems \ref{theo_main} and \ref{theo_mainGeo} still hold when the final time $t_f$ in OCP and OCP$^M$, respectively, is let free (moreover, due to $(\tilde A_2)$ the conditions $(p(t_f),p^0) \neq 0$ and $p^0_k \neq 0$ for every $k \ge 1$ are no longer required).}

\revThird{Due to lack of space, we do not report the proof of these latter results, referring the reader to \cite{BonalliArxiv2021} for further details.}

\section{Accelerating Convergence through \\ Indirect Shooting Methods}\label{sec_acceleratingProc}

An important result provided by Theorem \ref{theo_main} (and consequently by Theorem \ref{theo_mainGeo}) is the convergence of the sequence of the extremals (related to the sequence of convex subproblems) towards an extremal for the (penalized) original formulation. This can be leveraged to accelerate the convergence of Algorithm \ref{algo_SCP} by warm-starting indirect shooting methods \cite{betts1993,trelat2012}. Indirect shooting methods consist of replacing the original optimal control problem with a two-point boundary value problem formulated from the necessary conditions for optimality stated by the PMP. When indirect shooting methods succeed in converging to a locally-optimal solution, they converge very quickly (quadratically, in general). Nevertheless, they are very sensitive to initialization, which often presents a difficult challenge (see, e.g., \cite{trelat2012,bonalli2020}). In the following, with the help of Theorem \ref{theo_main}, we show how the initialization of indirect shooting methods may be bypassed by extracting information from the multipliers at each SCP iteration. The resulting indirect shooting methods may thus be combined with SCP to decrease (sometimes drastically decrease) the total number of iterations, given that the convergence rate of SCP is in general lower than that provided by shooting methods (see, e.g., \cite{diehl2019}). For the sake of clarity, we drop manifold-type constraints, knowing from Theorem \ref{theo_mainGeo} that the same reasoning can be applied to problems with such constraints.

From now on, without loss of generality we assume that every extremal that is mentioned below is normal, that is, by definition, $p^0 = -1$ (as mentioned earlier, this is a very mild requirement, see, e.g., \cite{ChitourJeanEtAl2008}). Assume that a time-discretized version of Algorithm \ref{algo_SCP} converges. In particular, due to the arguments in Section \ref{sec_guarantees}, we can assume that the convergence result stated in Theorem \ref{theo_main} applies to the sequence of KKT multipliers related to the time discretization of each convex subproblem LOCP$^{\Delta}_k$. For every $k \ge 1$, the KKT multiplier $\gamma^0_k$ that is related to the initial condition $x(0) = x^0$ approximates the initial value $p_k(0)$ of the extremal related to LOCP$^{\Delta}_k$ (see, e.g., \cite{GollmannKernEtAl2009}). Therefore, Theorem \ref{theo_main} implies that up to some subsequence, for every $\delta > 0$ there exists a $k_{\delta} \ge 1$ such that for every $k \ge k_{\delta}$, it holds that $\| p(0) - \gamma^0_k \| < \delta$, where $p$ comes from an extremal related to OCP. In particular, select $\delta > 0$ to be the radius of convergence of an indirect shooting method that we use to solve OCP (a rigorous notion of radius of convergence of an indirect shooting method may be inferred from the arguments in \cite{betts1993}). Any such indirect shooting method is then able to achieve convergence if initialized with $\gamma^0_k$, for $k \ge k_{\delta}$. In other words, we may stop SCP at iteration $k_{\delta}$ and successfully initialize an indirect shooting method related to the original (penalized) formulation with $\gamma^0_{k_{\delta}}$ to find a locally-optimal solution before SCP achieves full convergence, drastically reducing the number of SCP iterations used. Since in practice we do not have any knowledge of $\delta > 0$ and indirect shooting methods report convergence failures quickly, we can just run an indirect shooting method after every SCP iteration and stop whenever the latter converges (eventual convergence is ensured by the argument above). This acceleration procedure is summarized in Algorithm \ref{algo_SCP_acc}. Details concerning the implementation of indirect shooting methods are provided in the next section.

\SetKwInOut{Input}{Input}
\SetKwInOut{Output}{Output}
\SetKwInOut{Data}{Data}

\begin{algorithm}
\caption{Accelerated SCP} \label{algo_SCP_acc}
\Input{Guess trajectory $x_0$ and control $u_0$.}
\Output{Solution to OCP.}
\Data{Initial trust-region radius $\Delta_0 > 0$.}

\Begin
{
    $k = 0$, $\Delta_{k+1} = \Delta_k$, $\texttt{flag}=0$\\
    \While{($(u_k)_{k \in \mathbb{N}}$ \textnormal{has not converged}) \textnormal{or $\texttt{flag}=0$}}
    {
        Solve LOCP$^{\Delta}_{k+1}$ for \revThird{$(x_{k+1},u_{k+1})$}\\
        
        Solve an indirect shooting method on OCP to \revThird{$(x_{k+1},u_{k+1})$} initialized with the multiplier related to the constraint $x(0) = x^0$, and if successful, put $\texttt{flag}=1$ \\
        
        \revThird{$\Delta_{k+1} = \textbf{UpdateRule}(x_{k+1},u_{k+1},x_k,u_k)$}
        
            
        
            

        $k \gets k+1$\\

    }
    
    \Return{\revThird{$(x_{k-1},u_{k-1})$}}
}
\end{algorithm}

\section{Implementation and Numerical Experiments} \label{sec_numerics}





Next, we perform numerical experiments for the optimal control of a nonlinear system subject to obstacle-avoidance constraints \revThird{and provide implementation details. In particular, we}  demonstrate the performance of SCP and the gains from \revThird{the indirect shooting method}  acceleration procedure \revThird{in Alg. \ref{algo_SCP_acc}}. 
Albeit we present results for a fixed final time problem in this section, SCP can be applied to free final time OCPs with minor modifications; we refer to \url{https://arxiv.org/abs/2009.05038v3} for details and results.

\subsubsection{Problem formulation}
We consider 
a 3-dimensional non-holonomic Dubins car, with state $x = [r_x,r_y,\theta]\in\mathbb{R}^3$ and control input \revThird{$u =[v,\psi]\in\mathbb{R}^2$}. The dynamics are $\dot{x} \,{=}\, [v\cos(\theta), v\sin(\theta),k\revThird{\psi}]$, where \revThird{$k=0.1$ is the constant turning curvature. }
Starting from $x^0$, the objective of the problem is to reach the state $x_f$ while minimizing control effort $\int_0^{t_f}\revThird{(v(s)^2+\psi(s)^2)}\textrm{d}s$ and avoiding obstacles. Control bounds are set to \revThird{$U=[0,\bar{v}]\times[-\bar{\psi},\bar{\psi}]$ with $(\bar{v},\bar{\psi})=(0.5,1)$}. 
\revFourth{Note that this problem satisfies $(A_2)$, since $G(u(s))=v(s)^2+\psi(s)^2$ is $u$-strongly convex.} 
We consider problems with both fixed and free final angle $\theta(t_f)$. 
We 
consider $n_{\textrm{obs}}$ cylindrical obstacles of radius $\varepsilon_i$ centered at point $r_i\in\mathbb{R}^2$. For each obstacle, we set up an obstacle avoidance constraint using the smooth 
potential function $c_i:\mathbb{R}^2\rightarrow\mathbb{R}$, defined as
\begin{align}\label{eq:obs:potential}
c_i(r) = \begin{cases}
(\| r - r_i \|^2 - \varepsilon^2_i)^2, \quad   
    &\textrm{if }\ \| r - r_i \| < \varepsilon_i 
    \\
0,\quad
    & \mathrm{otherwise}
\end{cases}
,
\end{align}
where $r=[r_x,r_y]$. 
To incorporate these constraints within our problem formulation, we penalize them within the cost function and define OCP to minimize  $\int_0^{t_f}\big( \revThird{v(s)^2+\psi(s)^2}+\omega\sum_{i=1}^{n_{\textrm{obs}}}c_i(r) \big)\textrm{d}s$ \revThird{with $\omega=100$}, which is convex in $(r,u)$ and continuously differentiable. 
\revThird{Penalizing obstacle avoidance violations with this value for $\omega$ is sufficient to guarantee constraint satisfaction for the scenarios considered in the experiments.} 
This yields the following problem:
\begin{equation}\label{prob_dubins_free_tf_orig}
\quad
\begin{cases}
\underset{u }{\min} \ 
 \displaystyle \int^{ t_f}_0 \big( \revThird{v(s)^2+\psi(s)^2}+\omega\sum_{i}c_i(r(s)) \big) \ \mathrm{d}s 
 \\
 \dot{r}_x(s) = v\revThird{(s)} \cos \theta(s), 
 \quad
 \dot{r}_y(s) = v\revThird{(s)} \sin \theta(s), 
 \\
 \dot{\theta}(s) =  k \revThird{\psi}(s),
 \quad
 x(0) = x^0 , \quad  x(t_f) = x_f .
\end{cases}
\end{equation}

\subsubsection{Indirect shooting method} 
As described in Section \ref{sec_acceleratingProc} and Algorithm \ref{algo_SCP_acc}, 
the solution at each SCP iteration can be used to initialize an indirect shooting method for \eqref{prob_dubins_free_tf_orig}.  
Accordingly, we next derive the associated two-point boundary value  problem using the necessary conditions for optimality of the PMP. 
Assuming $p^0\,{=}\,{-}1$ (see Section \ref{sec_acceleratingProc}), the Hamiltonian
$
H_w(s,x,p,p^0,u) = p^{\top} f(s,x,u) + p^0f_w^0(s,x,u)
$ 
with $p\,{=}\,[p_x,p_y,p_{\theta}]$
is expressed as
\begin{align*}
H_w(x,u,p) 
&= v(p_x\cos\theta+p_y\sin\theta) +k \revThird{\psi} p_{\theta} 
\\
&\quad\ \  - \big(\revThird{v^2+\psi^2}+\omega\sum_{i=1}^{n_{\textrm{obs}}}c_i(r) \big)
.
\end{align*}
Applying the adjoint equation and the maximality condition of the PMP (Theorem \ref{theo_PMP}), we obtain the following relations: 
\begin{gather}
\label{eq:PMP_pdot}
\begin{split}
\dot{p}_x = &\ \omega\frac{\partial\big(\sum c_i(r)\big)}{\partial r_x} , \ \ 
\dot{p}_y = \omega\frac{\partial\big(\sum c_i(r)\big)}{\partial r_y}, 
\\
&\ \ \dot{p}_{\theta} = v(p_x \sin \theta - p_y \cos \theta), 
\ \  \ \text{and}
\end{split}
\end{gather}
$$
\revThird{v= \varphi_1(x,p) = \begin{cases}
\bar{v} & \text{if }\scalebox{0.925}{$\frac{(p_x\cos\theta+p_y\sin\theta)}{2}$} \ge \bar{v},
\\
\scalebox{0.925}{$\frac{(p_x\cos\theta+p_y\sin\theta)}{2}$} & \text{if }\scalebox{0.925}{$\frac{(p_x\cos\theta+p_y\sin\theta)}{2}$} \in (0,\bar{v}),
\\
0 & \text{if }\scalebox{0.925}{$\frac{(p_x\cos\theta+p_y\sin\theta)}{2}$} \le 0,
\end{cases}
}
$$
$$
\revThird{\psi= \varphi_2}(x,p) = \begin{cases}
\revThird{\bar{\psi}} & \text{if }\frac{p_{\theta} k}{2} \ge \revThird{\bar{\psi}},
\\
\frac{p_{\theta} k}{2} & \text{if }\frac{p_{\theta} k}{2} \in (-\revThird{\bar{\psi}},\revThird{\bar{\psi}}),
\\
-\revThird{\bar{\psi}} & \text{if }\frac{p_{\theta} k}{2} \le \revThird{\bar{\psi}}.
\end{cases}
$$
Further, the transversality conditions of the PMP for both problems with fixed and free final angle $\theta_f$ are shown in Figure \ref{fig:dubins:transversality}. 
\begin{figure}[!t]
\begin{center}
\begin{tabular}{|c|c|}
    \hline 
    \Tstrut 
    $\theta(t_f)$ fixed & 
    $\theta(t_f)$ free
    \Bstrut 
    \\
    \hline 
    \parbox{34mm}{
    \vspace{-0.5\baselineskip}
        \begin{subequations}
        \begin{align}
        \displaystyle
            \theta(t_f) - \theta_f = 0 
            \\
            H_{\omega}^*(t_f)=0
        \end{align}
        \end{subequations}
    \vspace{-1\baselineskip}
    }
    & 
    \parbox{34mm}{
    \vspace{-0.5\baselineskip}
        \begin{subequations}
        \begin{align}
            p_{\theta}(t_f) = 0
            \\
            H_{\omega}^*(t_f)=0
        \end{align}
        \end{subequations}
    \vspace{-1\baselineskip}
    }
    \\ 
    \hline 
\end{tabular}
\caption{Transversality conditions of the PMP for Dubins car, depending on whether or not the final angle is free. $H_{\omega}^*(t_f)$ denotes $H_{\omega}(x(t_f),\varphi(x(t_f),p(t_f)),p(t_f))$.}
\label{fig:dubins:transversality}
\vspace{-7mm}
\end{center}
\end{figure}
%
%
%
%
Based on these conditions, we define the shooting function $F: \mathbb{R}^3\times\mathbb{R}^3\rightarrow\mathbb{R}^{\revThird{3}}$ as:
\begin{align*}
    F_1(x(t_f),p(t_f)) &= r_x(t_f)-r_{x,f}
    \\
    F_2(x(t_f),p(t_f)) &= r_y(t_f)-r_{y,f}
    \\
    F_3(x(t_f),p(t_f)) &= \begin{cases}
        \theta(t_f)-\theta_{f}, & \text{if }\, \theta(t_f) \text{ fixed}
        \\
       p_{\theta}(t_f), & \text{if }\, \theta(t_f) \text{ free}
    \end{cases}
    .
\end{align*}
The PMP states that $F_i(x(t_f),p(t_f))\,{=}\,0$ for all $i\,{=}\,1,2,3$ for any locally-optimal trajectory. Thus, based on the conditions of the PMP, we set the following root-finding problem:
\begin{align*}
\textrm{Find}\ \revThird{p^0} \ \, \text{s.t.}\ \ &F_i(x(t_f),p(t_f))\,{=}\,0, \quad  i\,{=}\,1,2,3, 
\\
& \dot{x} = [\revThird{\varphi_1(x,p)}\cos\theta,\revThird{\varphi_1(x,p)}\sin\theta,k\revThird{\varphi_2}(x,p)], 
\\
& \dot{p} = \eqref{eq:PMP_pdot},
\ p(0) = p^0, 
\ x(0) = x^0.
\end{align*}
Given \revThird{$(x^0,p^0)$}, we obtain $x(t_f)$ and $p(t_f)$ by numerical integration of the dynamics and the adjoint equation. Then, given an initial guess, this problem can be solved using off-the-shelf root-finding algorithms, e.g., Newton's method. In this work, we use a fourth-order Runge-Kutta integration scheme to integrate differential equations and use the default trust-region method from the Julia \url{NLsolve.jl} package \cite{NLsolve2020} as the root-finding algorithm. 

As discussed in Section \ref{sec_acceleratingProc}, the success of solving this two-point boundary value problem is highly sensitive to the initial guess for \revThird{$p^0$}. 
To address this issue, we leverage the insights from Theorem \ref{theo_main}. Given a solution to LOCPP${}_{k+1}^\Delta$ strictly satisfying the trust-region constraints, we retrieve the KKT multiplier $\gamma_{k+1}^0$ associated with the initial condition $x(0)\,{=}\,x^0$. As discussed in Section \ref{sec_acceleratingProc}, $\gamma_{k+1}^0$ approaches $p^0$ as SCP converges to a locally-optimal trajectory. Thus, as described in Algorithm \ref{algo_SCP_acc}, we initialize the root-finding algorithm with \revThird{$\smash{\gamma_{k+1}^0}$} stemming from the solution of LOCPP${}_{k+1}^\Delta$. If a solution \revThird{$p^0$} to the root-finding problem is found, the corresponding candidate locally-optimal trajectory $(x,u\,{=}\,\varphi(x,p))$ has been found and Algorithm \ref{algo_SCP_acc} terminates. 

\subsubsection{Implementation details} 

To apply Algorithm \ref{algo_SCP} and \ref{algo_SCP_acc}, we start from \revThird{$\Delta_0=50$ and} we let $\Delta_{k+1}\leftarrow \revThird{0.5}\Delta_k$ to satisfy the assumptions of Theorem \ref{theo_main}. Note that different update rules are also possible \cite{BonalliCauligiEtAl2019}. 
We initialize SCP with a straight-line trajectory from $x^0$ to $x^f$, initialize all controls \revThird{inputs with $v(s)=\|r_f-r_0\|/t_f$ and $\psi(s)=0$ for $s\in[0,t_f]$}, and use a trapezoidal discretization scheme with  $N=51$ nodes. To check convergence of SCP, as a stopping criterion, we verify that $\int^{t_f}_0 \| u_{k+1}-u_k \|^2(s)\; \mathrm{d}s \leq 10^{-\revThird{2}}$. 
We solve each convexified problem using IPOPT. We release our implementation at \href{https://github.com/StanfordASL/jlSCP}{https://github.com/StanfordASL/jlSCP}. 
\begin{rmk}
In this setting, 
the assumptions of Theorem \ref{theo_main} (i.e., $p_k^0\neq 0$ for any $k\geq 1$ and $(p(t_f),p^0)\neq 0$) are automatically satisfied. Indeed, at each $k$-th SCP iteration, if IPOPT converges, then the critical point it finds  satisfies the KKT conditions, so that $p_k^0=-1\neq 0$ (up to numerical errors, see \cite{GollmannKernEtAl2009}). Moreover, if the shooting method (as described above) converges, then it automatically finds a solution with $p^0=-1$, so that $(p(t_f),p^0)\neq 0$.
\end{rmk}

\subsubsection{Results and discussion}\label{sec:results_discussion}

We evaluate our method in 100 randomized experiments. Denoting $\Unif(a,b)$ as the uniform probability distribution from $a\in\mathbb{R}$ to $b\in\mathbb{R}$, we set
\vspace{-3mm}

{\small
\begin{align*}
    &
    r_{x}^0\sim\Unif(-1,1),
    \  r_{y}^0\sim\Unif(-1,1),
    \ \ \theta^0\sim\Unif(-\pi,\pi),
    \\
    & 
    \smash{\tilde{\theta}_{xy}\sim\Unif(\theta^0{-}\frac{\pi}{4},\theta^0{+}\frac{\pi}{4}),
    \qquad\qquad\ \ 
    \theta^f\sim\Unif(\theta^0{-}\frac{\pi}{4},\theta^0{+}\frac{\pi}{4})},
    \\
    & 
    \smash{r_{x}^f\,{\sim}\, r_{x}^0 {+} (\revThird{5}{+}\Unif(0,3))\cos\tilde{\theta}_{xy},
    \ 
    r_{y}^f\,{\sim}\, r_{y}^0 {+} (\revThird{5}{+}\Unif(0,3))\sin\tilde{\theta}_{xy}},
    \\
    & 
    \epsilon_i{=}0.4, 
    \, n_{\textrm{obs}}{=}2,
    \, 
    \smash{r_{i,x}\,{\sim}\, 
    \Unif(\min(r_{x}^0,r_{x}^f){+}8\epsilon_i,
    \max(r_{x}^0,r_{x}^f){-}8\epsilon_i),}
\end{align*}
}%
and similarly for $r_{i,y}$.

\begin{figure}[t]
\centering
    \begin{minipage}{.65\linewidth}
        \centering
        {    
        \renewcommand{\arraystretch}{1.5}
        \small
        \begin{tabular}{|m{10.5mm}|c|c|}
            \hline 
             & 
            \multicolumn{2}{c|}{Avg. num. of SCP iterations}
            \\
            \cline{2-3}
            Problem &
            \hspace{-1mm}
            \begin{tabular}{m{14mm}}
            SCP only
            \end{tabular}
            \hspace{-4mm}
            &
            \hspace{-5.0mm}
            \begin{tabular}{m{14mm}}
            SCP+shooting
            \end{tabular}
            \hspace{-1.0mm}
            \\
            \hline 
            free $\theta_f$
                & \revThird{$4.12$}
                & \revThird{$3.54$}
            \\ 
            \hline 
            fixed $\theta_f$
                & \revThird{$4.52$}
                & \revThird{$3.78$}
            \\ 
            \hline 
        \end{tabular}
        }%
    \end{minipage}%
    \caption{Number of SCP iterations until convergence averaged over $100$ randomized experiments as described in Section \ref{sec:results_discussion}. }
    \label{fig:results_numbers}
%
%
%
    \begin{minipage}{.48\linewidth}
        \centering
        \includegraphics[width=1\linewidth,trim=30 0 40 0, clip]{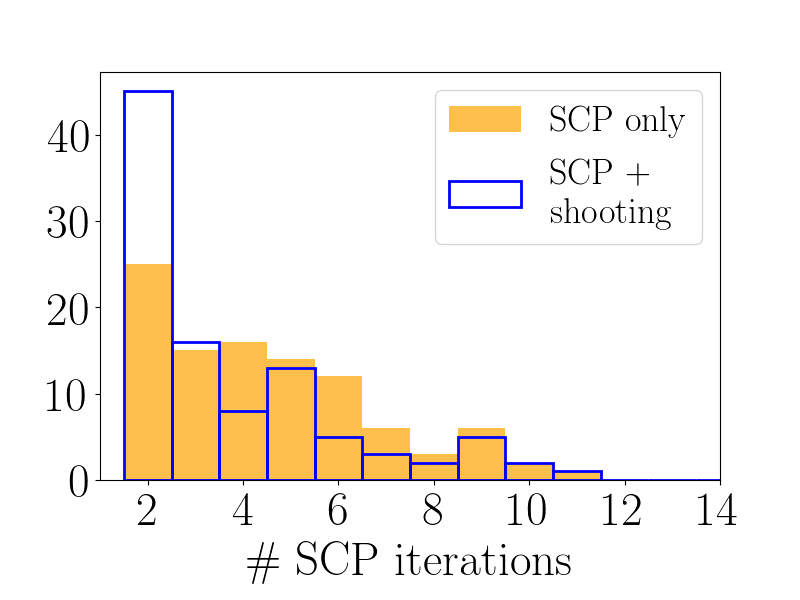}
    \end{minipage}%
    \hspace{1mm}
    \begin{minipage}{.48\linewidth}
        \centering
        \includegraphics[width=1\linewidth,trim=30 0 40 0, clip]{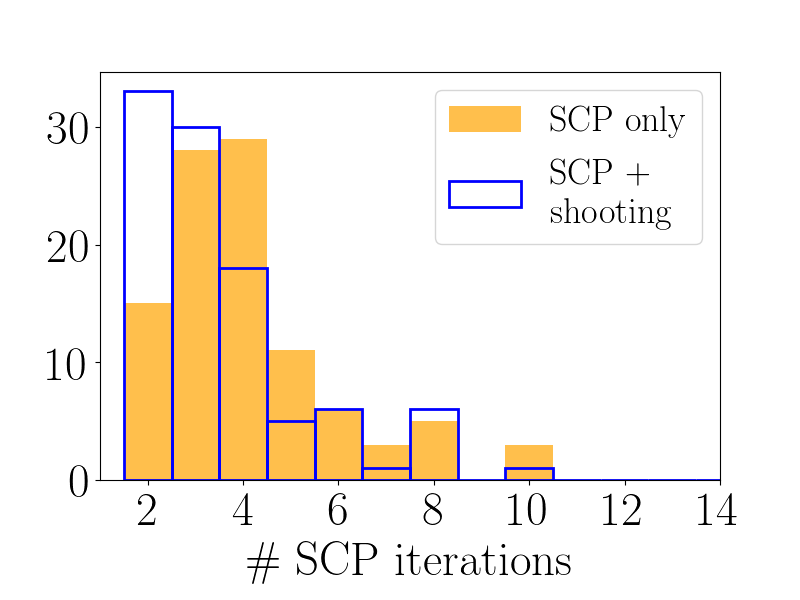}
    \end{minipage}%
    \caption{Results from randomized problems with 
    fixed final angle (left) and free final angle (right). These histograms show the number of SCP iterations until convergence for SCP only (orange, Algorithm \ref{algo_SCP}) and for shooting-accelerated SCP (blue, Algorithm \ref{algo_SCP_acc}).}
    \label{fig:histos}
\end{figure}

\begin{figure}[t]
    \begin{minipage}{.37\linewidth}
        \centering
        \includegraphics[width=1\linewidth,trim=0 0 0 0, clip]{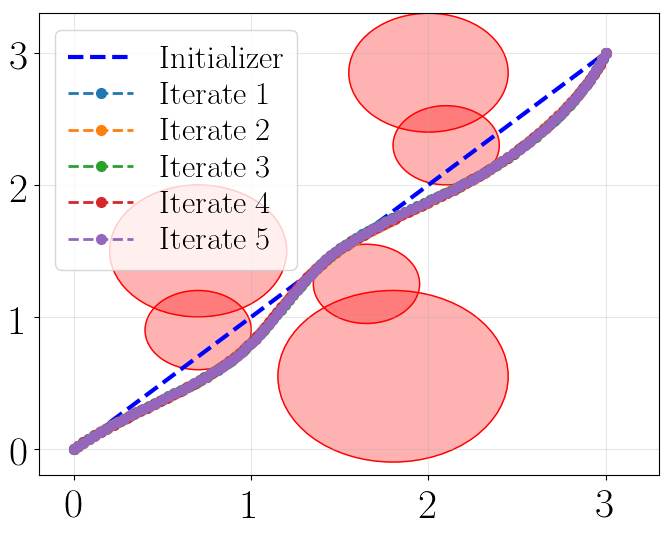}
    \end{minipage}%
    \begin{minipage}{.62\linewidth}
        \centering
        \includegraphics[width=1\linewidth,trim=0 0 0 0, clip]{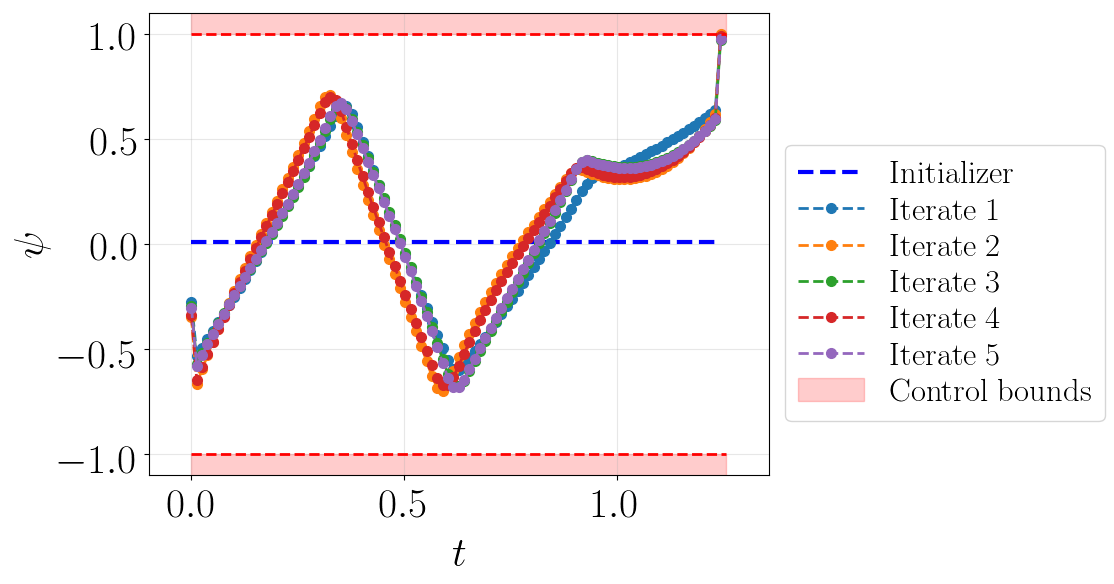}
    \end{minipage}%
    \caption{Example of trajectory using an infeasible straight-line initialization that passes through obstacles (left) and associated controls (right).}
    \label{fig:results_trajs}
\end{figure}

We consider the problems with both free and fixed final angle $\theta^f$. In 100\% of these scenarios, both SCP and the shooting-accelerated SCP converge successfully. 
Figure \ref{fig:results_trajs} shows that the initialization does not need to be feasible for SCP to converge successfully to a (candidate) locally-optimal trajectory avoiding obstacles and satisfying input constraints. 
Further, although the solution of the first iteration of SCP does not respect the nonlinear dynamics constraints, those become satisfied as the algorithm performs further iterations. 

Results in Figures \ref{fig:results_numbers} and \ref{fig:histos} demonstrate that leveraging the PMP through an indirect shooting method decreases the number of SCP iterations on average, significantly accelerating the algorithm. 
Indeed, SCP alone may require multiple iterations close to the optimal solution  before convergence. 
In contrast, once a good guess for \revThird{$p^0$} to initialize the root-finding algorithm is available, the indirect shooting method is capable of efficiently computing a (candidate) locally-optimal trajectory solving OCP. 
In the worst case where the number of SCP iterations until convergence $N_{\textrm{SCP}}$ is the same for both methods, 
which occurs if the guess for \revThird{$p^0$} is never within the radius of convergence of the shooting method at any SCP iteration, 
the computation time for Algorithm \ref{algo_SCP_acc} is $N_{\textrm{SCP}}\cdot(T_{\textrm{SCP}}+T_{\textrm{s-fail}})$, with $T_{\textrm{SCP}}$ being the time to convexify OCP and solve the resulting LOCP$^{\Delta}_{k+1}$, and $T_{\textrm{s-fail}}$ being the time for the root-finding algorithm to report convergence failure. \revFourth{In our non-optimized Julia implementation, \scalebox{0.95}{$T_{\textrm{s-fail}}=\revThird{25}\textrm{ms}$} and \scalebox{0.95}{$T_{\textrm{SCP}}=\revThird{242}\textrm{ms}$} on average, measured on a laptop equipped with a 2.60GHz Intel Core i7-6700 CPU with 8GB of RAM.} As $T_{\textrm{s-fail}}\,{\ll}\, T_{\textrm{SCP}}$ \revFourth{(see also computation times in \cite{Bonalli2018})}, 
there is little computational overhead in using accelerated-SCP over SCP only, and results in Figures \ref{fig:results_numbers} and \ref{fig:histos} demonstrate that leveraging the PMP significantly accelerates the optimization process. 
Finally, as \revThird{$p_k^0\neq 0$ holds at each SCP iteration} in $100\%$ of these scenarios \revThird{which we approximately check using the Lagrange multipliers $\gamma_k^0\approx p_k^0$}, from Theorem \ref{theo_main}, all trajectories are candidate locally-optimal solutions to OCPP${}_\omega$.

\section{Conclusion and Perspectives} \label{sec_conclusion}


In this paper, we analyze the convergence of SCP when applied to continuous-time non-convex optimal control problems, including  in the presence of manifold-type constraints. In particular, we prove that, up to some subsequence, 
SCP-based optimal control methods converge to a candidate locally-optimal solution for the original formulation. Under mild assumptions, our approach can be effortlessly leveraged to solve problems with manifold-type constraints. Finally, we leverage our analysis to accelerate the convergence of standard SCP-type schemes through indirect methods, and we investigate their performance via numerical simulations on a trajectory optimization problem with obstacles.

For future research, we plan to extend our approach to more general optimal control formulations, which for instance consider stochastic dynamics, risk functionals as costs, and probabilistic chance constraints. In addition, we plan to investigate particular parameters update rules which guarantee the convergence of the whole sequence of controls $(u_k)$ \revThird{(compare with Theorem \ref{theo_main} item 2)}. 
Finally, we plan to test the performance of our approach by means of hardware experiments on complex systems such as free-flyers and robotic manipulators.

\section*{Acknowledgment}

We thank Andrew Bylard, \revThird{Matthew Tsao}, and the anonymous referees for their careful review of this manuscript.

\bibliographystyle{IEEEtran}
\bibliography{ASL_papers,main}

\section{Appendix}

\begin{proof}[Proof of Lemma \ref{lemma_boundness}]
Fix $k \in \mathbb{N}$. For every $t \in \mathbb{R}_+$ where $x_{k+1}$ is defined, we may compute
{\scriptsize
\begin{align*}
    &\| x_{k+1}(t) \| \le \| x^0 \| + \left\| \int^t_0 \left( f_0(s,x_k(s)) + \sum^m_{i=1} u^i_{k+1}(s) f_i(s,x_k(s)) \right) \; \mathrm{d}s \right\| \\
    &\quad + \left\| \int^t_0 \bigg( \frac{\partial f_0}{\partial x}(s,x_k(s)) + \sum^m_{i=1} u^i_k(s) \frac{\partial f_i}{\partial x}(s,x_k(s)) \bigg) x_{k+1}(s) \; \mathrm{d}s \right\| \\
    &\quad + \left\| \int^t_0 \bigg( \frac{\partial f_0}{\partial x}(s,x_k(s)) + \sum^m_{i=1} u^i_k(s) \frac{\partial f_i}{\partial x}(s,x_k(s)) \bigg) x_k(s) \; \mathrm{d}s \right\| \\
    &\le \revThird{C_1} \left( 1 + \int^t_0 \| x_{k+1}(s) \| \; \mathrm{d}s \right) \\
    &\quad + \revThird{C_2} \sum^m_{i=0} \int_{\Big\{ s \in [0,t^{k+1}_f] : \ (s,x_k(s)) \in \textnormal{supp} f_i \Big\}} \left\| \frac{\partial f_i}{\partial x}(s,x_k(s)) \right\| \| x_k(s) \| \; \mathrm{d}s \\
    &\le \revThird{C_3} \left( 1 + \int^t_0 \| x_{k+1}(s) \| \; \mathrm{d}s \right) .
\end{align*}}
Here, $\textnormal{supp} f_i$ denotes the support of $f_i$, $i=0,\dots,m$. \revThird{The constants $C_1 , C_2 , C_3 \ge 0$} \revFourth{depend} on \revThird{$t_f$}, $f_i$, $i=0,\dots,m$, and $U$, and \revFourth{they come} from the compactness of $U$ and of $\textnormal{supp} f_i$ (and therefore of $\textnormal{supp} \frac{\partial f_i}{\partial x}$), $i=0,\dots,m$. By applying the Gr\"onwall inequality, we finally have that \revThird{$\| x_{k+1}(t) \| \le C_3 \exp(C_3 t_f)$ for every $t \in \mathbb{R}_+ \cap [0,t_f]$} where $x_{k+1}$ is defined. The compactness criterion for solution to ODE applies, and we infer that $x_{k+1}$ is defined in the entire interval \revThird{$[0,t_f]$}. Since the constant \revThird{$C_3 > 0$} does not depend on $k \in \mathbb{N}$, the trajectories $x_k$ are uniformly bounded in \revThird{$[0,t_f]$} and in $k \in \mathbb{N}$.
\end{proof}

\begin{IEEEbiography}[{\includegraphics[width=1in,height=1.25in,clip,keepaspectratio]{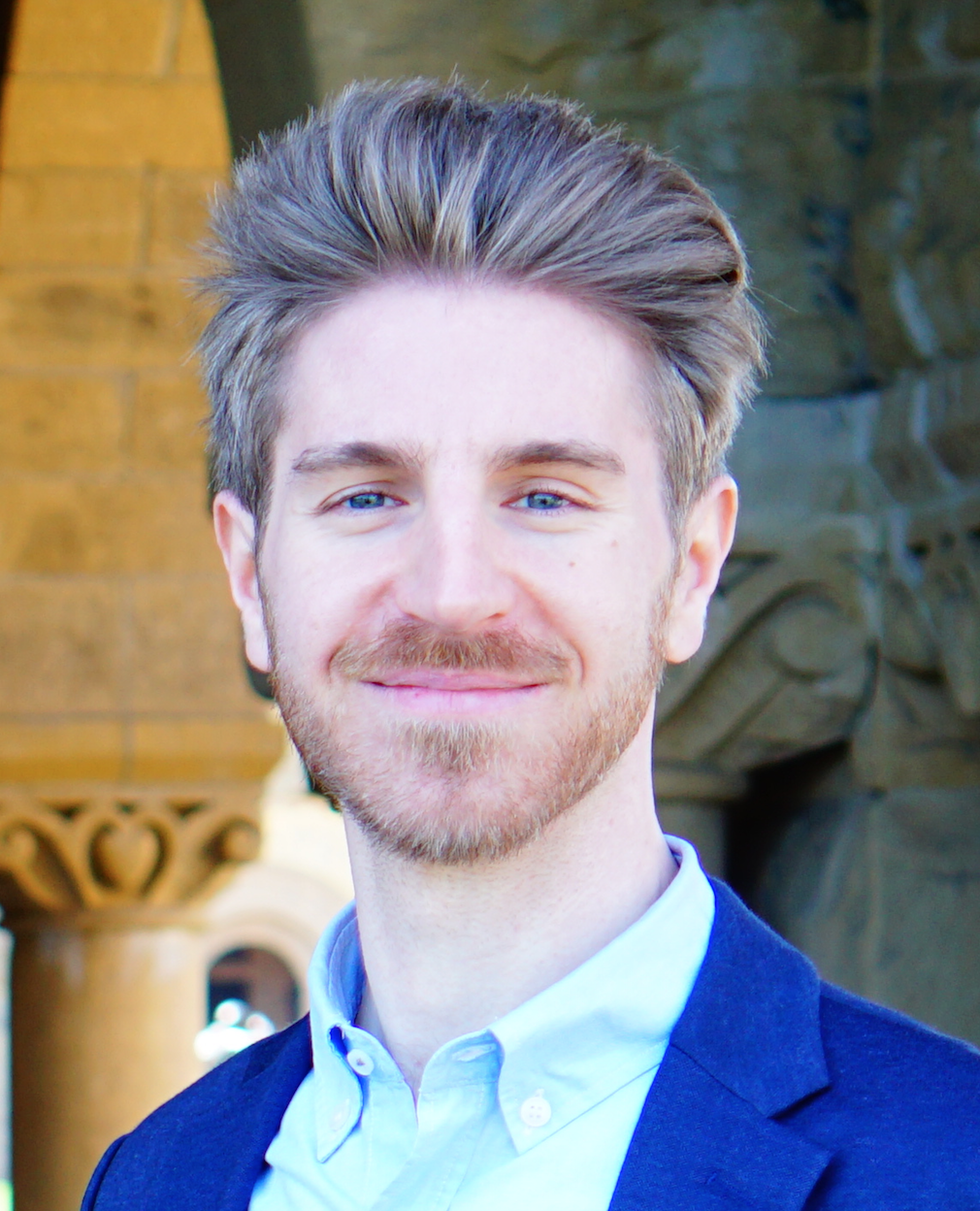}}]{Riccardo Bonalli} 
obtained his M.Sc. in Mathematical Engineering from Politecnico di Milano, Italy in 2014 and his Ph.D. in applied mathematics from Sorbonne Universite, France in 2018 in collaboration with ONERA - The French Aerospace Lab, France. He is a recipient of the ONERA DTIS Best Ph.D. Student Award 2018. He was a postdoctoral researcher with the Department of Aeronautics and Astronautics, Stanford University. Currently, Riccardo is a tenured CNRS researcher with the Laboratory of Signals and Systems (L2S), Universit\'e Paris-Saclay, Centre National de la Recherche Scientifique (CNRS), CentraleSup\'elec, France. His research interests concern theoretical and numerical robust optimal control with techniques from differential geometry, statistical analysis, and machine learning, and applications in aerospace systems and robotics.
\end{IEEEbiography}

\begin{IEEEbiography}[{\includegraphics[width=1in,height=1.25in,clip,keepaspectratio]{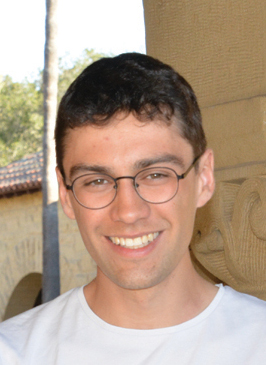}}]{Thomas Lew} received his BSc. degree in Microengineering from Ecole Polytechnique Federale de Lausanne in 2017, received his MSc. degree in Robotics from ETH Zurich in 2019, and is currently pursuing his Ph.D. degree in Aeronautics and Astronautics at Stanford University. His research focuses on the intersection between optimization, control theory, and machine learning techniques for aerospace applications and robotics.
\end{IEEEbiography}

\begin{IEEEbiography}[{\includegraphics[width=1in,height=1.25in,clip,keepaspectratio]{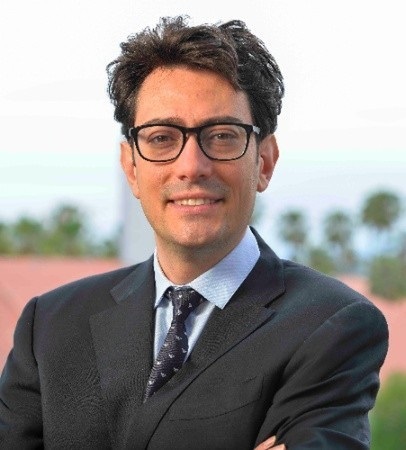}}]{Marco Pavone} 
is   an   Associate Professor of Aeronautics and Astronautics at Stanford University,  where he is the Director of the Autonomous Systems Laboratory. Before  joining  Stanford, he  was  a  Research  Technologist  within the  Robotics  Section  at  the  NASA  Jet Propulsion  Laboratory.   He  received  a Ph.D. degree in Aeronautics and Astronautics from the Massachusetts Institute of  Technology  in  2010.   His  main  research  interests  are  in the  development  of  methodologies  for  the  analysis,  design, and  control  of  autonomous  systems,  with  an  emphasis  on self-driving cars, autonomous aerospace vehicles, and future mobility systems.  He is a recipient of a number of awards, including  a  Presidential  Early  Career  Award  for  Scientists and Engineers, an ONR YIP Award, an NSF CAREER Award, and a NASA Early Career Faculty Award.  He was identified by the American Society for Engineering Education (ASEE) as one of America’s 20 most highly promising investigators under the age of 40.  He is currently serving as an Associate Editor for the IEEE Control Systems Magazine.
\end{IEEEbiography}

\end{document}